\documentclass[10pt,reqno]{amsart}
\usepackage{amsmath}
\usepackage{amssymb}
\usepackage{amsthm}
\usepackage{eepic,epic}
\usepackage{epsfig}
\usepackage{color}
\usepackage{graphicx}
\usepackage{array}

\newlength{\defbaselineskip}
\newcommand{\setlinespacing}[1]%
           {\setlength{\baselineskip}{#1 \defbaselineskip}}

\numberwithin{equation}{section}

\newtheorem{thm}{Theorem}[section]
\newtheorem{cor}[thm]{Corollary}
\newtheorem{lem}[thm]{Lemma}
\newtheorem{prop}[thm]{Proposition}

\theoremstyle{definition}
\newtheorem{defn}[thm]{Definition}

\theoremstyle{remark}
\newtheorem{rem}[thm]{Remark}
\numberwithin{equation}{section}

\newcommand{\R}{\mathbb R}

\newcommand{\bq}{\begin{equation}}
	\newcommand{\eq}{\end{equation}}

\newcommand{\lt}{\left}
\newcommand{\rt}{\right}

\newcommand{\ti}{\tilde}

\counterwithout{table}{section}

\makeatletter
\def\moverlay{\mathpalette\mov@rlay}
\def\mov@rlay#1#2{\leavevmode\vtop{%
		\baselineskip\z@skip \lineskiplimit-\maxdimen
		\ialign{\hfil$\m@th#1##$\hfil\cr#2\crcr}}}
\newcommand{\charfusion}[3][\mathord]{
	#1{\ifx#1\mathop\vphantom{#2}\fi
		\mathpalette\mov@rlay{#2\cr#3}
	}
	\ifx#1\mathop\expandafter\displaylimits\fi}
\makeatother

\def\XXint#1#2#3{{\setbox0=\hbox{$#1{#2#3}{\int}$ }
		\vcenter{\hbox{$#2#3$ }}\kern-.6\wd0}}

\begin{document}

\title[The energy-critical INLS with strong singularity]
{The 3D energy-critical inhomogeneous nonlinear Schr\"{o}dinger equation with strong singularity}

\author{Yoonjung Lee}

\thanks{This research is supported by
	NRF-2022R1A2C1002820, NRF-2022R1I1A1A01068481 and RS-2024-00406821.}


\address{Department of Mathematics, Yonsei University, 50 Yonsei-Ro, Seodaemun-Gu, Seoul 03722, Republic of Korea}
\email{yjglee@yonsei.ac.kr}

\begin{abstract}
In this paper, we study the Cauchy problem for the 3D energy-critical inhomogeneous nonlinear Schr\"odinger equation(INLS) $$i\partial_{t}u+\Delta u=\pm|x|^{-\alpha}|u|^{4-2\alpha}u$$ with strong singularity $3/2\leq \alpha<2$.
The well-posedness problem is well-understood for $0<\alpha<3/2$, but the case $3/2\leq \alpha<2$ has remained open so far.
We address the local/small data global well-posedness result for $3/2\leq \alpha<11/6$ by improving the inhomogeneous Strichartz estimates on the weighted space. 
\end{abstract}

\maketitle

\section{Introduction}
We consider the Cauchy problem of the energy-critical inhomogeneous nonlinear Schr\"odinger equation (INLS) in three dimensions
\begin{equation}\label{3DINLS}
\left\{
\begin{aligned}
&i \partial_{t} u + \Delta u = \lambda |x|^{-\alpha} |u|^{4-2\alpha} u
\\
&u(x, 0)=u_0 \in H^1(\R^3)
\end{aligned}
\right.
\end{equation}
with strong singularity $3/2\leq\alpha<2$. 
Here, $u(t, x)$ is a complex-valued function in $ I\times  \mathbb{R}^3 $ for a time interval $I \in \mathbb{R}$.
The value of $\lambda$ is determined as $1$ or $-1$, corresponding to the \textit{defocusing} and \textit{focusing} model, respectively.
This type of the equation physically arises in nonlinear optics with spatially dependent interaction to describe the propagation of laser beam in the inhomogeneous medium (see e.g.\cite{B, BPVT}).  
The equation \eqref{3DINLS} has a scaling symmetry $u(t, x)\mapsto 
\delta^{\frac{1}{2}} u(\delta^2 t, \delta x)$
for $\delta>0$ and the energy norm $\dot{H}^1(\R^3)$ of the solution remains invariant under the scaling.
This means a \textit{energy-critical} case. 

The well-posedness problem of \eqref{3DINLS} was studied in depth with $\alpha=0$ over the past several decades, see e.g. \cite{B2, CW, CKSTT, KM, T, TV} based on the classical Strichartz estimates.
In recent years, many researchers have attempted to extend it to the case $\alpha>0$.
Relying on the classical Strichartz estimates, the energy-critical case of \eqref{3DINLS} was first dealt in \cite{CHL} with $0<\alpha<4/3$.
For more extended range $0<\alpha<3/2$, weighted spaces and Lorentz space approaches were introduced in \cite{AK, AT, LS}, also see \cite{CL}.
Lately, in \cite{CCF}, such results were recovered on the classical $L^p$ setting by using the generalized Leibniz rule.  
However, there is no existence result of \eqref{3DINLS} for $3/2\leq \alpha<2$ up to now.
We would like to refer to \cite{CHL, GM, GX, LYZ, LZ, P} where the scattering results for \eqref{3DINLS} were achieved under $0<\alpha<3/2$.

In this paper, we are interested in the unsolved case $3/2\leq \alpha<2$. 
A main technical issue in this case is that as the nonlinear term become more singular, a more regular solution space is required.  
To treat the issue, we consider the weighted space approach because it is not only to handle efficiently the singularity on the nonlinear term, but also to expect the gain of regularity in virtue of the smoothing effect. 
We will extend the validity of exponents where inhomogeneous Strichartz esimates hold in the weighted $L^p$ space.

\subsection{Extended Strichartz estimates on weighted spaces}
We first recall the weighted version of the homogeneous Strichartz estimate, focusing on the three dimensions.
It was first discovered in \cite{OR} for diagonal case, and it was generalized in \cite{KLS, KLS2} shown as follows.

\begin{thm}(\cite{KLS2}) \label{thm_homo}
	Let $\gamma>0$ and $-1/2<\sigma<3/2$.
	Then estimate 
	\begin{equation}\label{1}
		\| e^{it\Delta}f\|_{L_t^q L_x^r(|x|^{-r\gamma})} \lesssim \|f\|_{\dot{H}^{\sigma}}.
	\end{equation}
	holds 
	if $(q, r)$ satisfies 
	\begin{equation}\label{admiss}
		0\leq \frac{1}{q} \leq \frac{1}{2}, \quad \frac{\gamma}{3}< \frac{1}{r} \leq \frac{1}{2},
		\quad \frac{1}{q}<\frac{3}{2}(\frac{1}{2}-\frac{1}{r})+\gamma,\quad
		\sigma=3(\frac{1}{2}-\frac{1}{r})-\frac{2}{q}+\gamma.
	\end{equation}
\end{thm}

The condition \eqref{admiss} in Theorem \ref{thm_homo} is almost optimal except for $1/r\leq 1/2$, see \cite{KLS2}. 
We define a class $\mathcal{AD}_{\gamma, \sigma}$, for convenience, as a set of the exponents $(q, r)$ satisfying \eqref{admiss} for $\gamma$ and $\sigma$.  
Note that $\mathcal{AD}_{\gamma, 0}$, sometimes using the notation $\mathcal{AD}_{\gamma}$, is the pair set of the reduced condition 
\[
\frac{1}{6}+\frac{\gamma}{3} \leq \frac{1}{r} \leq \frac{1}{2}, \quad \frac{2}{q}=3(\frac{1}{2}-\frac{1}{r})+\gamma.
\]

As a direct application of the Christ-Kiselev Lemma \cite{CK} and $TT^{\ast}$-argument, Theorem \ref{thm_homo} implies that
\begin{equation}\label{inhomo}
	\left\| \int_0^t e^{i(t-s) \Delta} F(s) ds \right\|_{L_t^{q} L_x^{r}(|x|^{-r\gamma})} \lesssim  \|F \|_{L_t^{\tilde{q}'} L_x^{\tilde{r}'}(|x|^{\tilde{r}'\tilde{\gamma}})}
\end{equation}
also holds in the non-sharp case $1/q+1/\tilde{q}<1$ if $(q, r) \in \mathcal{AD}_{\gamma, \sigma}$ and $(\tilde{q}, \tilde{r})\in \mathcal{AD}_{\tilde{\gamma}, -\sigma}$ with $-1/2<\sigma<1/2$.
But an interesting fact is that the non-weighted version of \eqref{inhomo} is valid for more wider range than the that of homogeneous ones, see \cite{F2, KS, V2}.
According to previous studies, 
we will improve the inhomogeneous Strichartz estimates on the weighted spaces.
We consider the class $\mathcal{AC}_{\gamma, \sigma}$, which is more wider class than $\mathcal{AD}_{\gamma, \sigma}$ for $\gamma, \sigma>0$. 
\begin{defn}
	Let $0< \gamma<1$ and $-1/2<\sigma<1/2$. We say $(q, r)\in \mathcal{AC}_{\gamma, \sigma}$ if
	\begin{equation}\label{acad}
		0\leq\frac{1}{q} \leq1,\qquad \frac{\gamma}{3}<\frac{1}{r}<\frac{1}{2}+\frac{\gamma}{3}, \qquad \frac{1}{q}<3(\frac{1}{2}-\frac{1}{r})+\gamma
	\end{equation}
	and 
	\begin{equation}\label{sc_1}
	\frac{2}{q}=3\lt(\frac{1}{2}-\frac{1}{r}\rt)+\gamma-\sigma.
	\end{equation}
\end{defn}
Then we explore an extension of the weighed inhomogeneous Strichartz estimate in the following theorem, whose proof is also available for higher dimensional cases.

\begin{figure}
		\label{pic}
	\centering
	\begin{tabular}{cc}
	\includegraphics[width=0.4\textwidth]{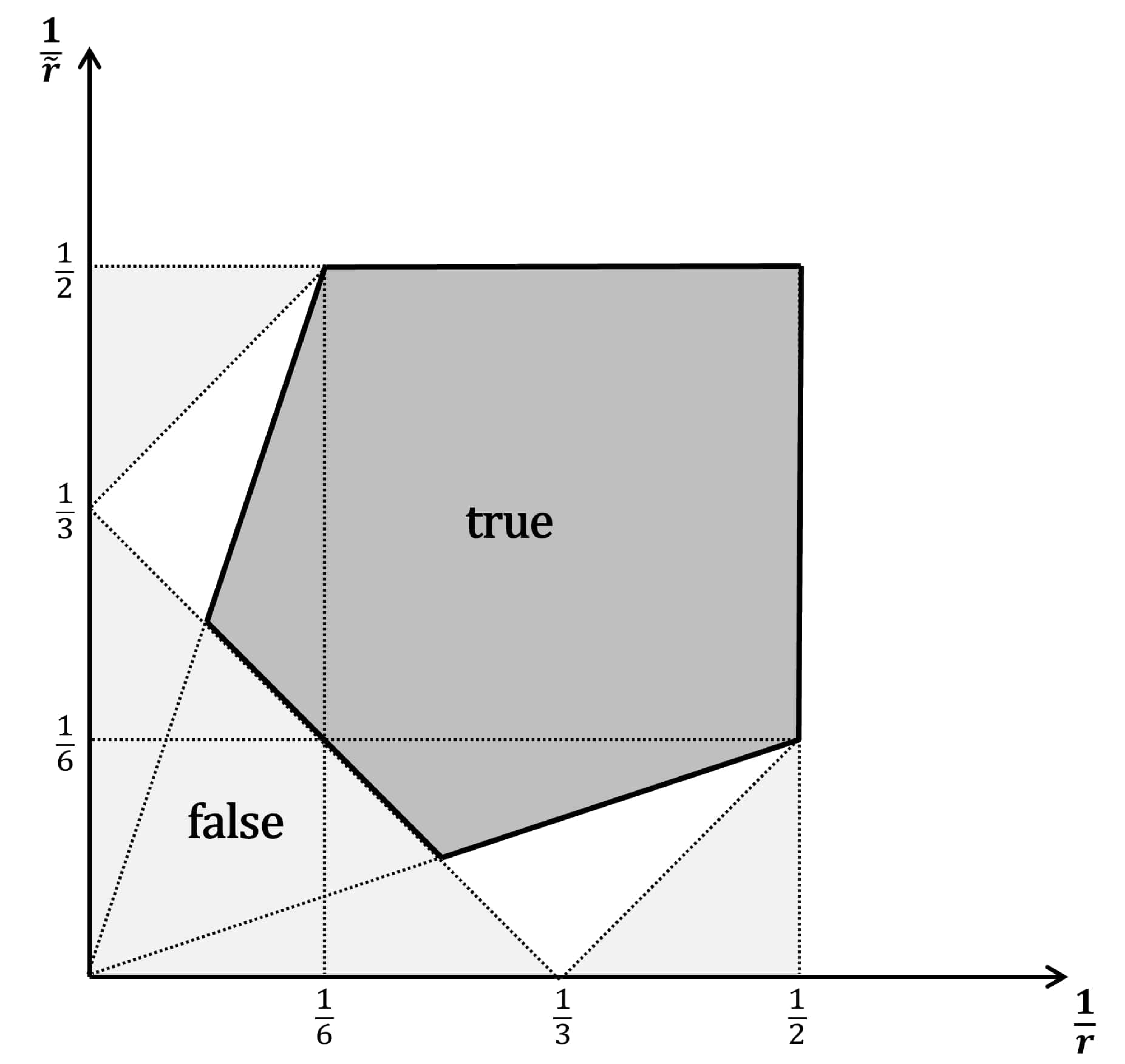} 
	\includegraphics[width=0.45\textwidth]{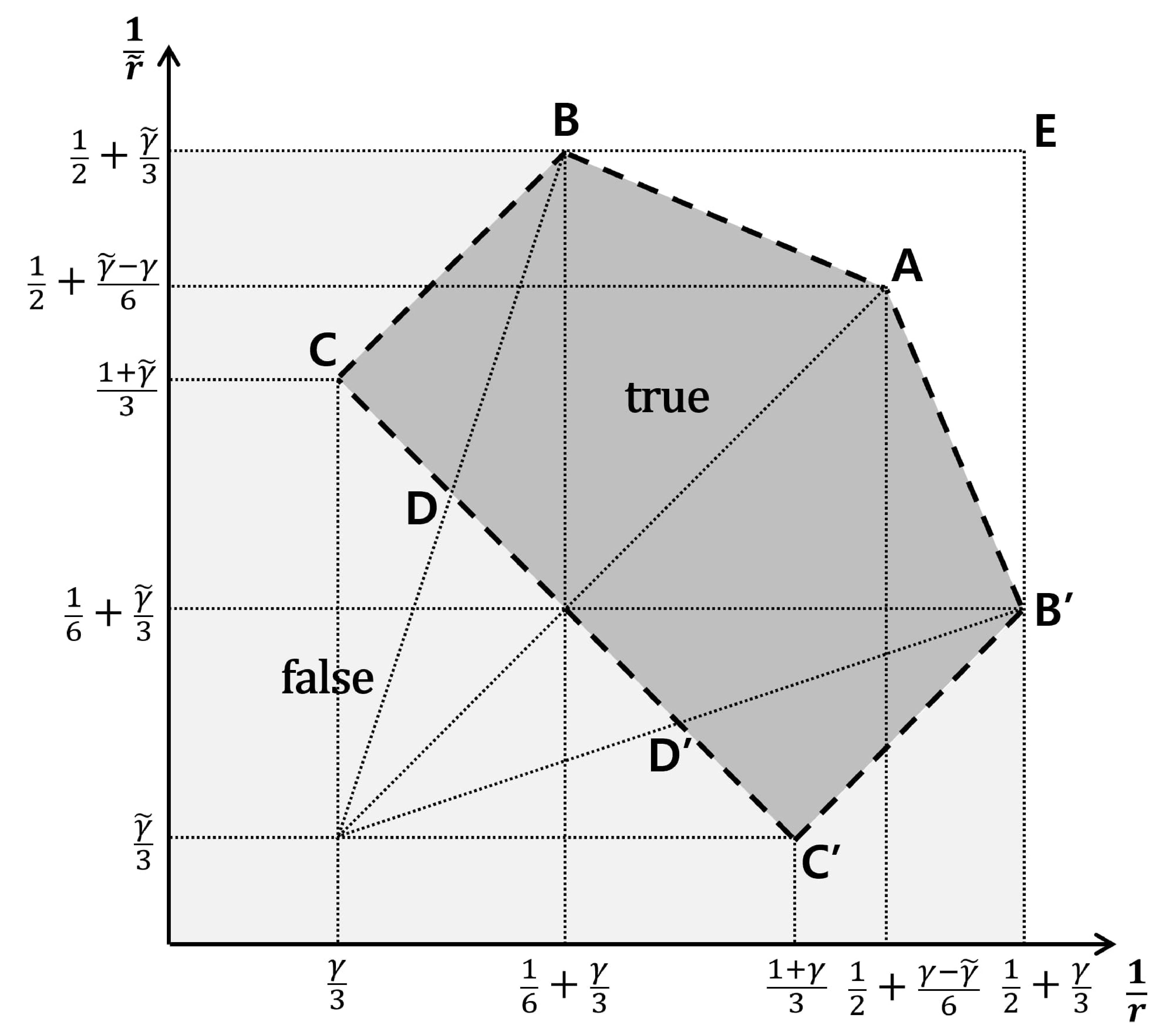} 
	\end{tabular}\\
	\footnotesize{The case $\gamma=\tilde{\gamma}=0\qquad \qquad  \qquad \qquad \qquad \qquad \text{The case } \gamma, \tilde{\gamma}>0$}
	\caption{The validity of $(1/r, 1/\tilde{r})$ on \eqref{inhomo}.}
\end{figure}

\begin{thm}\label{thm_inhomo}
	Let $0<\gamma, \tilde{\gamma}<1$.
	The estimate \eqref{inhomo} holds
	if $(q,r)\in \mathcal{AC}_{\gamma, \sigma}$ and $(\tilde{q}, \tilde{r})\in \mathcal{AC}_{\tilde{\gamma}, -\sigma}$ satisfying 
\bq\label{main_cond_1}
\frac{1}{q}+\frac{1}{\tilde{q}} < 1, \quad  \quad \lt| \frac{1}{\tilde{r}}-\frac{\tilde{\gamma}}{3} -\frac{1}{r}+\frac{\gamma}{3} \rt| <\frac{1}{3}.
\eq 
and one of the following cases hold:\newline
	$\bullet$  when $\frac{1}{r}-\frac{\gamma}{3}=\frac{1}{\tilde{r}}-\frac{\tilde{\gamma}}{3}$
		\bq\label{main_cond_11}
		\frac{1}{6}+\frac{\gamma}{3}<\frac{1}{r} \leq \frac{1}{2}-\frac{\tilde{\gamma}-\gamma}{6},
		\eq
	$\bullet$ when $\frac{1}{r}-\frac{\gamma}{3}<\frac{1}{\tilde{r}}-\frac{\tilde{\gamma}}{3}$
		\begin{gather}
			\label{main_cond_21}\frac{1}{r}-\frac{\gamma}{3} \leq -\frac{2-\gamma-\tilde{\gamma}}{\gamma+\tilde{\gamma}}\left( \frac{1}{\tilde{r}}-\frac{\tilde{\gamma}}{3}\right)+\frac{3-\gamma-\tilde{\gamma}}{3(\gamma+\tilde{\gamma})}, \\
			\label{main_cond_22}  
			\frac{1}{q}<\frac{3}{2}+3\lt(\frac{1}{r}-\frac{\gamma}{3}\rt) -3\lt(\frac{1}{\tilde{r}}-\frac{\tilde{\gamma}}{3} \rt), \quad \tilde{\gamma}<\frac{1}{2},
		\end{gather}
	$\bullet$ when $\frac{1}{r}-\frac{\gamma}{3}>\frac{1}{\tilde{r}}-\frac{\tilde{\gamma}}{3}$
		\begin{gather}
			\label{main_cond_31} \frac{1}{\tilde{r}}-\frac{\tilde{\gamma}}{3} \leq -\frac{2-\gamma-\tilde{\gamma}}{\gamma+\tilde{\gamma}}\left(\frac{1}{r}-\frac{\gamma}{3}\right)+\frac{3-\gamma-\tilde{\gamma}}{3(\gamma+\tilde{\gamma})},\\
			\label{main_cond_32}
			\frac{1}{\tilde{q}}<\frac{3}{2}+3\lt(\frac{1}{\tilde{r}}-\frac{\tilde{\gamma}}{3} \rt)-3\lt(\frac{1}{r}-\frac{\gamma}{3}\rt), \quad  \gamma<\frac{1}{2}.
		\end{gather}
\end{thm}

We would like to give some sharpness in the theorem.
For the class $\mathcal{AC}_{\gamma, \sigma}$ and $\mathcal{AC}_{\tilde{\gamma}, -\sigma}$,
the condition \eqref{acad} consists of necessary conditions where the estimate \eqref{inhomo} hold, whereas \eqref{sc_1} is taken account of for convenience of application. 
In fact, such type of \eqref{sc_1} can be replaced by the sharp condition
\begin{equation}\label{sc}
	\frac{1}{q}+\frac{1}{\tilde{q}}=\frac{3}{2}(1-\frac{1}{r}-\frac{1}{\tilde{r}})+\frac{\gamma+\tilde{\gamma}}{2}
\end{equation}
in Theorem \ref{thm_inhomo}.
Our proof will be shown with the scaling condition \eqref{sc}.
On the other hand, $1/q+1/\tilde{q}\leq 1$ is also a necessary condition, due to the time-translation invariant property of \eqref{inhomo}.
We provide the theorem for the non-sharp case $1/q+1/\tilde{q}<1$, because the sharp case (line $CC'$ in Figure \ref{pic}) does not improve the range of $\alpha$ on the well-posedness result to \eqref{3DINLS}, 
Meanwhile, the second condition of \eqref{main_cond_1} corresponding to the region between two lines $BC$ and $B'C'$ is needed for \eqref{inhomo} to hold. For more details on the necessary conditions, see Proposition \ref{prop}.

There are additional comments on the remaining ones \eqref{main_cond_11}-\eqref{main_cond_32}.
The upper bound of \eqref{main_cond_11} in the diagonal case (line $AE$) seems to be restricted, but it is the best possible range obtained by the decay estimates, Lemma \ref{lem1}, where $1/r+1/\tilde{r} \leq 1$. 
The equality case of \eqref{main_cond_21} and \eqref{main_cond_31} represent to lines $AB$ and $AB'$ in turn.
By the Christ-Kiselev Lemma and $TT^{\ast}$-argument as shown before, one could extend the validity of $(1/r, 1/\tilde{r})$ beyond the line $AB$ when $\tilde{\gamma}< \gamma$ and the line $AB'$ when $\gamma<\tilde{\gamma}$ though, it is not a crucial part to derive the well-posedness result. Finally, the restrictions \eqref{main_cond_22} and \eqref{main_cond_32} arise from the extension of $(1/r, 1/\tilde{r})$ to the optimal range, i.e. the second one in \eqref{main_cond_1}.
In other words, if we additionally assume  
\[
\frac{1}{\tilde{b}}-\frac{\tilde{\gamma}}{3} < 3\lt( \frac{1}{b}-\frac{\gamma}{3}\rt) \quad \text{and} \quad \frac{1}{b}-\frac{\gamma}{3} < 3\lt( \frac{1}{\tilde{b}}-\frac{\tilde{\gamma}}{3}\rt),
\]
corresponding to lines $BD$ and $B'D'$, \eqref{main_cond_22} and \eqref{main_cond_32} can be removed.

As mentioned before, it was studied in \cite{F2, KS, V2} that the non-weighted version of \eqref{inhomo} is true  
for the pair $(q, r)$ satisfying the case $\gamma=0$ of \eqref{acad} including with $r=2, \infty$ and similarly for $(\tilde{q}, \tilde{r})$
if $$
\frac{1}{q}+\frac{1}{\tilde{q}}=3(1-\frac{1}{r}-\frac{1}{\tilde{r}})
$$
and $(1/r, 1/\tilde{r})$ lies in the dark gray colored area of the left figure in Figure \ref{pic}.  
From this fact, one may expect that the weighted inhomogeneous estimates hold for a pentagon $EBDD'B'E$ in Figure \ref{pic}, translating it along the $1/r$-axis by $\gamma/3$ and the $1/\tilde{r}$-axis by $\tilde{\gamma}/3$, respectively.
However, two points are observed from Theorem \ref{thm_inhomo}. 
First, from the arguments in the previous works it is not obvious that \eqref{inhomo} holds for $1/r$ and $1/\tilde{r}$ beyond $1/2$. 
We split the space domain into regions around/far from the origin, and handle each domain differently. 
Even if it is not a significant region to the application, 
we would like to mention that the case $1/r+1/\tilde{r}>1$ is quite challenging.
Second, the smoothing effect causes the pentagon to spread out to the optimal region compared to the non-weighted setting.
More precisely, the weighted $L^p$ spaces allow us to obtain the gain of regularity for the Schr\"odinger flow and it facilitate to obtain the triangles $BCD$ and $B'C'D'$ which will play a crucial role to deal with the unsolved problem of the INLS model in 3D energy-critical case.

\subsection{Well-posedness results of \eqref{3DINLS}}
We present the local and small data global well-posedness of \eqref{3DINLS} when $3/2\leq \alpha<11/6$. 
Taking advantage of Theorem \ref{thm_inhomo}, we will deal with the nonlinear term with strong singularity. 

For appropriate $(q, r)\in \mathcal{AD}_{\gamma}$ and $(\tilde{q}, \tilde{r}) \in \mathcal{AD}_{\tilde{\gamma}}$,
one can get 
\bq\label{nonlinear}
\|\nabla(|x|^{-\alpha} |u|^{4-2\alpha} u)\|_{L_t^{\tilde{q}'}L_x^{\tilde{r}'} (|x|^{\tilde{r}'\tilde{\gamma}})} \lesssim \|\nabla u\|_{L_t^{q}L_x^{r} (|x|^{-r\gamma})}^{5-2\alpha}
\eq
by the H\"older inequality and the weighted Sobolev inequality.
For the exponents $(q, r)$ and $(\ti{q}, \ti{r})$, then it should be satisfied that
\[
1=(5-2\alpha)\lt(\frac{1}{r}-\frac{\gamma}{3}\rt)+\frac{1}{\ti{r}}-\frac{\ti{\gamma}}{3}-\frac{4-3\alpha}{3}, \quad \frac{1+\gamma}{3}<\frac{1}{r}.
\]
But the condition $\tilde{q}\ge2$ on the class $\mathcal{AD}_{\tilde{\gamma}}$ leads to the existence result with the restriction $\alpha<3/2$, see \cite{LS}.
However, the improved inhomogeneous Strichartz estimates (Theorem \ref{thm_inhomo}) allow us to choose $(\tilde{q}, \tilde{r})\in \mathcal{AC}_{\tilde{\gamma}, 0}=\mathcal{AC}_{\tilde{\gamma}}$ with no restriction of $\tilde{q}\ge 2$.
This is the reason why we refine the inhomogeneous Strichartz estimates.

We consider convenient spaces $W=W(I\times \mathbb{R}^3)$ and $\dot{W}^1=\dot{W}^1(I \times \mathbb{R}^3)$ defined as the closure of the test functions under the following norms
\begin{equation*}
	\|f\|_{W} := \||x|^{-\gamma_0} f\|_{L_t^{q_0}(I;L_x^{ r_0})}\quad \text{and} \quad  \|f\|_{\dot{W}^1} := \|\nabla f\|_{W}
\end{equation*}
with 
\[
\frac{1}{q_0}=\frac{3}{2(17-6\alpha)}+\frac{3\delta}{2},\quad  
\frac{1}{r_0}=\frac{1}{2}-\frac{1}{17-6\alpha}+\frac{\gamma_0}{3}-\delta
\]
for sufficiently small $\delta>0$.
One sees $(q_0, r_0)\in \mathcal{AD}_{\gamma_0}$.
We consider $N=N(I\times \mathbb{R}^3)$ and $\dot{N}^1=\dot{N}^1(I\times \mathbb{R}^3)$ as the nonlinearity spaces equipped with the norm
\begin{equation*}
	\|F\|_{N} := \||x|^{\ti{\gamma}_0}F \|_{L_t^{\ti{q}_0'}(I;L_x^{\ti{r}_0'})} \quad  \text{and} \quad	\|F\|_{\dot{N}^1} := \|\nabla  F\|_{N}
\end{equation*}
where $(\tilde{q}_0, \tilde{r}_0)$ satisfies
\[
\frac{1}{\ti{q}_0}=\frac{1}{2}+\frac{1}{17-6\alpha}-\frac{3\delta(5-2\alpha)}{2},\quad  
\frac{1}{\tilde{r}_0}=\frac{1}{6}-\frac{2}{3(17-6\alpha)}+\frac{\ti{\gamma}_0}{3}+\delta(5-2\alpha).
\]
It is easy to check $(\ti{q}_0, \ti{r}_0) \in \mathcal{AC}_{\ti{\gamma}_0}$ for sufficiently small $\delta>0$.
These norms on the spaces $W, \dot{W}^1, N$ and $\dot{N}^1$ are particularly chosen to be well adapted to control the nonlinearity, see Lemma \ref{lem_nonlinear}.

We also define the weighted Strichartz spaces $S_{\gamma}=S_\gamma(I \times \mathbb{R}^n)$ and $\dot{S}_{\gamma}^1=\dot{S}^1_\gamma(I \times \mathbb{R}^n)$ for $0<\gamma<1/2$ equipped with
$$ \|f\|_{S_\gamma}
:= \sup \|f\|_{L_{t}^q(I;L_{x}^r(|x|^{-r\gamma})) } \quad \text{and} \quad
\|f\|_{\dot{S}^1_\gamma} :=  \| \nabla f\|_{S_{\gamma} },$$
respectively. Here, the supremum is taken over $(q, r)\in \mathcal{AD}_{\gamma}$ satisfying 
\bq\label{cond_qr}
\frac{1}{6}+\frac{\gamma}{3}+\frac{2}{3(17-6\alpha)}-\delta(5-2\alpha)<\frac{1}{r}<\frac{1}{2}+\frac{\gamma}{3}-\frac{1}{17-6\alpha}+\frac{3\delta(5-2\alpha)}{2}.
\eq 
It is also noted that  $S_{\gamma} \subseteq W$ and $\dot{S}_{\gamma}^1 \subseteq \dot{W}^1$ in the sense that
\bq\label{relation_WS}
\|f\|_{W}\lesssim \|f\|_{S_{\gamma}}, \qquad \|f\|_{\dot{W}^1}\lesssim \|f\|_{\dot{S}_{\gamma}^1}.
\eq

Now we present our local well-posedness result of \eqref{3DINLS} in the unsolved case $3/2\leq \alpha<11/6$.

\begin{thm}(Local well-posedness)\label{thm}
Let $I$ be a compact time interval and let $u_0 \in H^{1}(\mathbb{R}^3)$ be such that
	\begin{equation}\label{as}
	 \|e^{it\Delta} u_0\|_{\dot{W}^{1}(I \times \mathbb{R}^3)} \leq \eta
	\end{equation}
	for some $0<\eta\leq \eta_0$ where $\eta_0>0$ is a small constant.
	If $3/2\leq \alpha<11/6$, then there exist a unique solution $u$ to \eqref{3DINLS} in $ C_t H_x^{1}(I\times \mathbb{R}^3) \cap \dot{S}_{\gamma}^1(I\times \mathbb{R}^3)$.
	Moreover, we have
	\begin{align}
	& \|u\|_{\dot{W}^{1}(I \times \mathbb{R}^3)} \lesssim \eta,\label{w1}\\
	&\|u\|_{S_{\gamma} (I \times \mathbb{R}^3)} \lesssim \|u_0\|_{L^2}, \label{w3} \\
	&\|u\|_{\dot{S}^1_{\gamma}(I \times \mathbb{R}^3)} \lesssim  \|u_0\|_{\dot{H}^{1}} +\eta^{\beta+1}. \label{w2}
	\end{align}
	Furthermore, the continuous dependence on the initial data holds with associated solution $\tilde{u} \in C_t H_x^{1}(I\times \mathbb{R}^3) \cap \dot{S}_{\gamma}^1(I\times \mathbb{R}^3)$ such as
	\begin{equation}\label{co}
	\| u(t)- \tilde{u}(t) \|_{S_{\gamma}} \lesssim \|u_0 -\tilde{u}_0\|_{L^2}.
	\end{equation}
\end{thm}

We also obtain the small data global well-posedness for the energy-critical INLS.
In the critical case, the local solution exists in a time interval depending on the data $u_0$ itself and not on its norm.
For this reason, $\|u_0\|_{H^1}$ is assumed to be small and the global well-posedness result is stated as follows.

\begin{cor}(Global well-posedness for small data).
Let $u_0\in H^{1}(\mathbb{R}^3)$ be such that
	\begin{equation*}
\| u_0\|_{H^{1}(\mathbb{R}^3)}	 \leq \eta_0
	\end{equation*}
	for small $\eta_0>0$.
If $3/2\leq \alpha<11/6$, then there exists a unique global solution of the problem \eqref{3DINLS} with
	$u \in C_tH_x^{1}([0,\infty)\times \mathbb{R}^3) \cap \dot{S}^1_{\gamma}([0,\infty)\times \mathbb{R}^3)$.
	Moreover, we have
	\begin{equation*}
 \|u\|_{\dot{W}^1(\mathbb{R}\times \mathbb{R}^3)}, \|u\|_{\dot{S}_{\gamma}^1 (\mathbb{R}\times \mathbb{R}^3)} \lesssim \|u_0\|_{\dot{H}^1}
	\end{equation*}
	\begin{equation*}
	\|u\|_{S_{\gamma}(\mathbb{R}\times \mathbb{R}^n)} \lesssim  \|u_0\|_{L^2}.
	\end{equation*}
\end{cor}

\begin{rem}
	Recently, the authors of \cite{CCF} showed that the INLS model is ill-posed in $H^1(\mathbb{R}^3)$ for $3/2\leq \alpha<2$ and the nonlinearity index $\beta\in 2\mathbb{Z}$. In fact, they proved that there is no $C^{1+\beta}$-flow in $H^1(\mathbb{R}^3)$ generated by the INLS, strongly relying on the classical Strichartz space. However, from Theorem \ref{thm}, it seems that the weighted Strichartz space is enable to construct the $C^{1+\beta}$-flow to obtain the well-posedness result in general situations.  
\end{rem}

The remainder of this paper is organized as follows. 
In Section \ref{sec2}, we extend the inhomogeneous Strichartz estimates on the weighted spaces. 
For this, we establish the time-localized estimates (Proposition \ref{lem2}) valid for a broader range by investigating the time-decay estimates (Lemma \ref{lem1}) and the smoothing effect of the Schr\"odinger flow on weighted spaces.
We then apply a bilinear interpolation argument to the localized estimates, leading to the global inhomogeneous Strichartz estimates, which hold in the extended range. 
These estimates are employed to build up the well-posedness of the 3D energy-critical model with strong singularity in Section \ref{sec3}.
Section \ref{sec4} is devoted to discussing the sharpness of the theorem, which plays a crucial role in deriving the well-posedness result.

Throughout this paper, the letter $C$ stands for a positive constant which may be different
at each occurrence.
We write $A \lesssim B$ to indicate that there exists a constant $c>0$ such that $A \leq cB$, and $A \sim B$ if there exists $c>1$ such that $A/c \leq B \leq cA$.  
Additionally, $\hat{f}$ and $\mathcal{F}(f)$ stand for the Fourier transform of $f$, and $\mathcal{F}^{-1}$ represents the inverse Fourier transform. 
The weighted spaces $L_{t}^q L_{x}^r(|x|^{-r\gamma})$ is defined by
\begin{equation*}
	\|f\|_{L_{t}^q L_{x}^r(|x|^{-r\gamma})}=  \bigg( \int_{\mathbb{R}} \bigg(\int_{\mathbb{R}^n}  |x|^{-r\gamma} |f(x)|^r  dx \bigg)^{\frac{q}{r}} dt\bigg)^{\frac{1}{q}}.
\end{equation*}

\section{Refined Strichartz estimates}\label{sec2}

In this section, we prove Theorem \ref{thm_inhomo}.
For this, we will show a strong type of the bilinear estimate
\begin{equation}\label{d2}
	\int_{\mathbb{R}} \int_{-\infty}^t \left\langle e^{-is\Delta } F(s),\, e^{-it\Delta } G(t) \right\rangle ds dt \lesssim \|F\|_{L_t^{q'} L_x^{r'}(|x|^{r'\gamma})} \|G\|_{L_t^{\tilde{q}'} L_x^{\tilde{r}'}(|x|^{\tilde{r}' \tilde{\gamma}})}
\end{equation}
under the same conditions as in Theorem \ref{thm_inhomo}.
Then we divide the domain of integral in $t$ on the left-hand side of \eqref{d2} into two parts, $t\ge0$ and $t<0$.
By a change of the variable $t \rightarrow -t$, $t<0$ can be changed to $t\ge0$.
The part $t\ge 0$ is viewed as $[0,t)=(0,t)\cap [0, \infty)$. Replacing $F$ with $\chi_{[0, \infty)}(s) F$, \eqref{d2} implies 
\begin{equation*}
	\int_{\mathbb{R}} \int_{0}^t \left\langle e^{-is\Delta } F(s),\, e^{-it\Delta } G(t) \right\rangle ds dt \lesssim \|F\|_{L_t^{q'} L_x^{r'}(|x|^{r'\gamma})} \|G\|_{L_t^{\tilde{q}'} L_x^{\tilde{r}'}(|x|^{\tilde{r}' \tilde{\gamma}})}
\end{equation*}
which is equivalent to \eqref{inhomo} by the duality.

We decompose $\Omega=\{ (s, t)\in \mathbb{R}^2 :s<t \}$ dyadically away from $s=t$
thanks to the Whitney decomposition (see \cite{F,S});
let us consider the family of dyadic squares $Q_j$ whose side length is $2^j$ for $j\in \mathbb{Z}$.
Then, each square $Q =I\times J$ in $Q_j$ has a property
\begin{equation}\label{wht}
	2^j\sim |I| \sim |J| \sim \textnormal{dist}(I, J),
\end{equation}
and the squares $Q$ are pairwise essentially disjoint so that $\Omega = \cup_{j\in \mathbb{Z}} \cup_{Q \in Q_j} Q$.
Thus \eqref{d2} can be rewritten as
\begin{align}\label{d3}
	\sum_{j\in \mathbb{Z}} \sum_{Q \in Q_j}  |T_{j, Q}(F, G)| \lesssim \|F\|_{L_t^{q'} L_x^{r'}(|x|^{r'\gamma})} \|G\|_{L_t^{\tilde{q}'} L_x^{\tilde{r}'}(|x|^{\tilde{r}' \tilde{\gamma}})}
\end{align}
where $T_{j, Q}$ is a time-localized operator of $2^j$-size
\begin{equation*}
	T_{j, Q}(F, G) :=  \int_{t\in J} \int_{s \in I}  \left\langle e^{-is\Delta} F(s),\ e^{-it\Delta} G(t) \right\rangle ds dt.
\end{equation*}
Then we turn to showing the localized estimates
\begin{equation}\label{eq2}
	|T_{j, Q}(F, G)| \lesssim 2^{-j\beta( a, b, \tilde{a}, \tilde{b})} \|F\|_{L_t^{a'}(I; L_x^{b'}(|x|^{b'\gamma}))} \|G\|_{L_t^{\tilde{a}'}(J; L_x^{\tilde{b}'}(|x|^{\tilde{b}' \tilde{\gamma}}))}
\end{equation}
with 
\begin{equation}\label{be}
	\beta( a, b, \tilde{a}, \tilde{b}) := -\lt(\frac{1}{a}+\frac{1}{\tilde{a}} \rt) +\frac{3}{2}\lt(1-\frac{1}{b}-\frac{1}{\tilde{b}}\rt)+\frac{\gamma+\tilde{\gamma}}{2}.
\end{equation}

\begin{figure}
	\centering
	\includegraphics[width=0.6\textwidth]{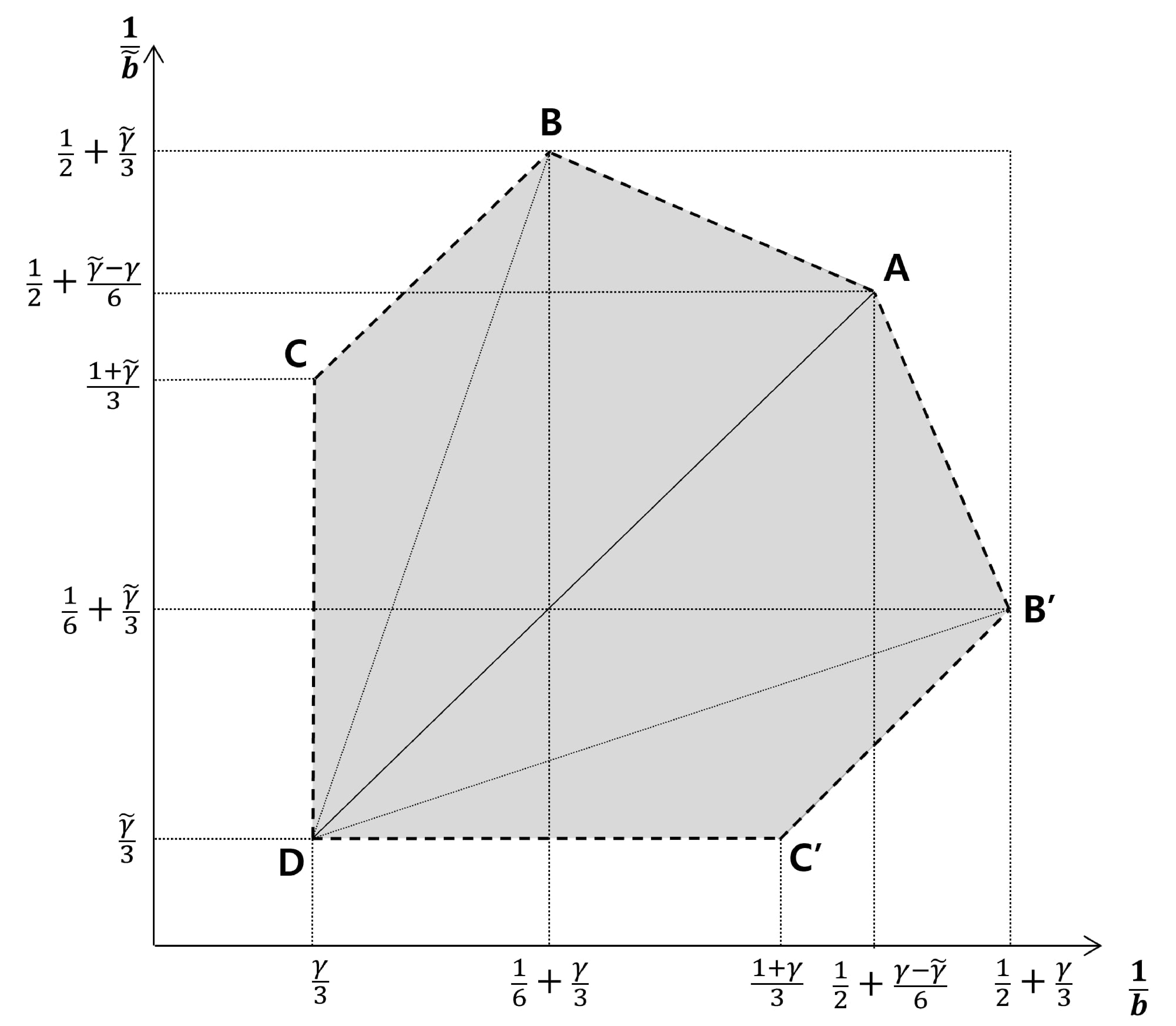} 
	\label{pic1}
	\caption{The validity of $(1/b, 1/\tilde{b})$ on \eqref{eq2}. }
\end{figure}

The following proposition provides the validity of $a, b, \tilde{a}, \tilde{b}$ where the time localized inhomogeneous estimate \eqref{eq2} holds in weighted spaces.

\begin{prop}\label{lem2}
	Let $0< \gamma, \tilde{\gamma} < 1$.
	Then the estimate \eqref{eq2} holds for $j \in \mathbb{Z}$ if $1\leq a, \tilde{a}\leq \infty$ ,
	\bq
	\begin{aligned}
		\label{i00}  
		\frac{\gamma}{3}<\frac{1}{b},\quad \frac{\tilde{\gamma}}{3}<\frac{1}{\tilde{b}}, \quad \lt| \frac{1}{\tilde{b}}-\frac{\tilde{\gamma}}{3} -\frac{1}{b}+\frac{\gamma}{3}\rt|<\frac{1}{3} 
	\end{aligned}
	\eq 
	and one of the following sets of conditions additionally holds:\newline
	$\bullet$ diagonal case $\overline{AD}$
	\bq\label{i01}
	\frac{\gamma}{3}< \frac{1}{b} \leq \frac{1}{2}+\frac{\gamma-\tilde{\gamma}}{6},
	\eq
	$\bullet$ upper-side of  $\overline{AD}$ (i.e. $\square ABCD$)
	\begin{gather}
		\label{i02}\frac{1}{b}-\frac{\gamma}{3} < -\frac{2-\gamma-\tilde{\gamma}}{\gamma+\tilde{\gamma}}\left( \frac{1}{\tilde{b}}-\frac{\tilde{\gamma}}{3}\right)+\frac{3-\gamma-\tilde{\gamma}}{3(\gamma+\tilde{\gamma})}, \\
		\label{i03} \frac{3}{2}\lt( \frac{1}{\tilde{b}}-\frac{\tilde{\gamma}}{3}-\frac{1}{b}+\frac{\gamma}{3}\rt) \leq \frac{1}{a} \leq \frac{1}{\tilde{a}'}\leq 1,\\
		\label{i04}  \frac{3}{2}\lt( \frac{1}{\tilde{b}}-\frac{\tilde{\gamma}}{3}\rt) -\frac{9}{2} \lt(\frac{1}{b}-\frac{\gamma}{3}\rt) <\frac{1}{\tilde{a}} ,  \quad \tilde{\gamma}<\frac{1}{2} ,
	\end{gather}
	$\bullet$ lower-side of  $\overline{AD}$ (i.e. $\square AB'C'D$)
	\begin{gather}
		\frac{1}{\tilde{b}}-\frac{\tilde{\gamma}}{3} < -\frac{2-\gamma-\tilde{\gamma}}{\gamma+\tilde{\gamma}}\left(\frac{1}{b}-\frac{\gamma}{3}\right)+\frac{3-\gamma-\tilde{\gamma}}{3(\gamma+\tilde{\gamma})},\nonumber\\
		 \frac{3}{2}\lt( \frac{1}{b}-\frac{\gamma}{3} -\frac{1}{\tilde{b}}+\frac{\tilde{\gamma}}{3}\rt) \leq \frac{1}{\tilde{a}}\leq \frac{1}{a'}\leq 1,\nonumber\\
	\frac{3}{2}\lt( \frac{1}{b}-\frac{\gamma}{3}\rt) -\frac{9}{2} \lt( \frac{1}{\tilde{b}}-\frac{\tilde{\gamma}}{3}\rt) < \frac{1}{a} , \quad \gamma<\frac{1}{2}.\nonumber
	\end{gather}
	(See Figure \ref{pic1}.)
\end{prop}

To sum it over $Q\in Q_j$, we note that for each $I$ there are at most two squares $Q$ in $Q_j$ with interval $J$ satisfying \eqref{wht}
and so 
it follows that
\begin{align*}
	& \sum_{Q\in Q_j}\left|T_{j, Q}(F, G) \right|\\
	&\quad\quad \leq  2^{-j\beta(a,b;\tilde{a}, \tilde{b})}\sum_{Q \in Q_j}  \|F \|_{L_t^{a'}(I; L_x^{b'}(|x|^{b'\gamma}))} \|G\|_{L_t^{\tilde{a}'} (J;L_x^{\tilde{b}'} (|x|^{\tilde{b}' \tilde{\gamma}}))} \\
	&\quad\quad \leq  2^{-j\beta(a,b; \tilde{a}, \tilde{b})} \left(\sum_{Q \in Q_j } \|F \|_{L_t^{a'}(I; L_x^{b'} (|x|^{b' \gamma}))}^{a'}  \right)^{\frac{1}{a'}} \cdot \left(\sum_{Q \in Q_j} \|G\|_{L_t^{\tilde{a}'} (J;L_x^{\tilde{b}'} (|x|^{\tilde{b}' \tilde{\gamma}}))}^{\tilde{a}'} \right)^{\frac{1}{\tilde{a}'}}\\
	&\quad\quad \leq  2^{-j\beta(a,b; \tilde{a}, \tilde{b})} \|F \|_{L_t^{a'} L_x^{b'}(|x|^{b' \gamma})}  \|G \|_{L_t^{\tilde{a}'} L_x^{\tilde{b}'}(|x|^{\tilde{b}' \tilde{\gamma}})}
\end{align*}
by the Cauchy-Schwartz inequality and H\"older inequality for $1/a+1/\tilde{a}\leq 1$.

Next, we take a summation over $j \in \mathbb{Z}$ to deduce \eqref{d3}. 
But it cannot be directly summable because $\beta$ is not always positive.
To gain the summability over $j \in \mathbb{Z}$, we will apply the bilinear interpolation argument.
For non sharp case $1/q+1/\tilde{q}<1$, we fix $r$ and $\tilde{r}$.
Then we choose perturbed exponents of $q, \tilde{q}$ such as
$$ \frac{1}{q_0}=\frac{1}{q} -\varepsilon_0,\quad \frac{1}{q_1}=\frac{1}{q}+2\varepsilon_0,\quad \frac{1}{\tilde{q}_0}=\frac{1}{\tilde{q}}-\varepsilon_0 ,\quad \frac{1}{\tilde{q}_1}=\frac{1}{\tilde{q}}+2\varepsilon_0$$
for a sufficiently small $\varepsilon_0>0.$
Then by making use of Proposition \ref{lem2}, we get
\begin{align*}
	|T_j (F,G)|  \lesssim 2^{-j\beta(q_0, r; \tilde{q}_0, \tilde{r})}& \|F\|_{L_t^{q_0'} L_x^{r'}(|x|^{r'\gamma})}  \|G\|_{L_t^{\tilde{q}_0'} L_x^{\tilde{r}'}(|x|^{\tilde{r}'\tilde{\gamma}})}, \cr
	|T_j (F,G)|  \lesssim 2^{-j\beta(q_0, r; \tilde{q}_1, \tilde{r})}& \|F\|_{L_t^{q_0'} L_x^{r'}(|x|^{r'\gamma})}  \|G\|_{L_t^{\tilde{q}_1'} L_x^{\tilde{r}'}(|x|^{\tilde{r}'\tilde{\gamma}})}, \cr
	|T_j (F,G)|  \lesssim 2^{-j\beta(q_1, r; \tilde{q}_0, \tilde{r})}& \|F\|_{L_t^{q_1'} L_x^{r'}(|x|^{r'\gamma})}  \|G\|_{L_t^{\tilde{q}_0'} L_x^{\tilde{r}'}(|x|^{\tilde{r}'\tilde{\gamma}})}
\end{align*}	
in which $\beta(q_0, r; \tilde{q}_0, \tilde{r}) = 2\varepsilon_0$ and $\beta(q_0, r; \tilde{q}_1, \tilde{r})=\beta(q_1, r; \tilde{q}_0, \tilde{r}) = - \varepsilon_0$.

With a bilinear vector-valued operator
\begin{equation*}
	B(F, G)=\left\{  T_{j}(F, G) \right\}_{j \in \mathbb{Z}},
\end{equation*}
we rewrite the above three estimates as
\begin{align*}
	& \|B(F, G)\|_{\ell_{\infty} ^{2\varepsilon}} \lesssim \|F\|_{L_t^{q_0'} L_x^{r'}(|x|^{r'\gamma})}  \|G\|_{L_t^{\tilde{q}_0'} L_x^{\tilde{r}'}(|x|^{\tilde{r}'\tilde{\gamma}})}\cr
	& \|B(F, G)\|_{\ell_{\infty} ^{-\varepsilon}} \lesssim \|F\|_{L_t^{q_0'} L_x^{r'}(|x|^{r'\gamma})}  \|G\|_{L_t^{\tilde{q}_1'} L_x^{\tilde{r}'}(|x|^{\tilde{r}'\tilde{\gamma}})}\cr
	& \|B(F, G)\|_{\ell_{\infty} ^{-\varepsilon}} \lesssim \|F\|_{L_t^{q_1'} L_x^{r'}(|x|^{r'\gamma})}  \|G\|_{L_t^{\tilde{q}_0'} L_x^{\tilde{r}'}(|x|^{\tilde{r}'\tilde{\gamma}})}.
\end{align*}
Here, $\ell_p^a$ is a weighted sequence space defined for $a \in \mathbb{R}$ and $1 \leq p \leq \infty$ with the norm
\begin{align*}
	\|\{ x_j \}_{j \ge 0} \|_{\ell_p^a} =
	\begin{cases}
		(\sum_{j \ge 0} 2^{jap} |x_j|^p )^{\frac{1}{p}}, & \mbox{if } p \neq \infty \\
		\sup_{j \ge 0} 2^{ja} |x_j|, & \mbox{if } p =\infty.
	\end{cases}
\end{align*}

By applying the following lemma with $s =1$, $p=q'$, $q=\tilde{q}'$ and $\theta_0=\theta_1=1/3$, we now get
\begin{align}\label{end3}
	B \ : \ (A_0, A_1)_{\frac{1}{3},\,q'}\, \times\,(B_0, B_1)_{\frac{1}{3},\,\tilde{q}'}\rightarrow(\ell_{\infty} ^{ 2\varepsilon}, \ell_{\infty} ^{-\varepsilon})_{\frac{2}{3},\,1}
\end{align}
where $A_i=L_t^{q_i'} L_x^{r'}(|x|^{r' \gamma})$, $B_i=L_t^{\tilde{q}_i'} L_x^{\tilde{r}'} (|x|^{\tilde{r}' \tilde{\gamma}})$ for $i=0,1$.

\begin{lem}(\cite{BL}, Section 3.13, Exercise 5(b))
	For $i=0,1$, let $A_i , B_i , C_i$ be Banach spaces and let $T$ be a bilinear operator such that $T: A_0 \times B_0 \rightarrow C_0$, $T : A_0 \times B_1 \rightarrow C_1$, and $T : A_1 \times B_0 \rightarrow C_1$. Then one has
	\begin{equation*}
		T : (A_0,A_1)_{\theta_0, p} \times (B_0, B_1)_{\theta_1, q} \rightarrow (C_0, C_1)_{\theta,1}
	\end{equation*}
	if $0< \theta_i < \theta=\theta_0+\theta_1 <1$ and $1/p + 1/q \geq 1$ for $1 \leq p,q \leq \infty$.
	Here, $(\cdot\,,\cdot)_{\theta,p}$ denotes the real interpolation functor.
\end{lem}

In the end, we use the real interpolation space identities;
\begin{equation*}
	(\ell_{\infty}^{s_{0}},\, \ell_{\infty}^{s_{1}})_{\theta, 1} = \ell_{1}^{s},\quad s_{0} \neq s_{1}
\end{equation*}
for $s_0, s_1 \in \mathbb{R}$, $s=(1-\theta)s_{0} +\theta s_{1}$, we see
\begin{equation}\label{ide}
	(A_0, A_1)_{\frac{1}{3},\,q'}\, = L_t^{q'} L_x^{r'}(|x|^{r' \gamma}) \quad \text{and} \quad
	(B_0, B_1)_{\frac{1}{3},\,\tilde{q}'}\, = L_{t}^{\tilde{q}'} L_x^{\tilde{r}'} (|x|^{\tilde{r}' \tilde{\gamma}})
\end{equation}
and $(\ell_{\infty} ^{2\varepsilon}, \ell_{\infty} ^{-\varepsilon})_{\frac{2}{3},\,1} = \ell_{1} ^0 $. 
Combining \eqref{end3} and \eqref{ide}, hence it follows
\begin{equation*}
	B\ :\  L_t^{q'} L_x^{r'} (|x|^{r' \gamma}) \times L_t^{\tilde{q}'} L_x^{\tilde{r}'} (|x|^{\tilde{r}' \tilde{\gamma}}) \rightarrow \ell_1^0.
\end{equation*}
which is equivalent to \eqref{d3}.

The requirements in Theorem \ref{thm_inhomo} immediately follows from Proposition \ref{lem2}.
In fact, letting $(a, b)=(q, r)$ and $(\tilde{a}, \tilde{b})=(\tilde{q}, \tilde{r})$, the last condition of \eqref{i00} becomes the last one of \eqref{main_cond_1}. In the diagonal case where $1/r-\gamma/3=1/\tilde{r}-\tilde{\gamma}/3$, \eqref{i01} implies \eqref{main_cond_11} because of $1/q+1/\tilde{q}<1$ and \eqref{sc}. Meanwhile, for the upper-side of the diagomal case, \eqref{i02} and \eqref{i04} show \eqref{main_cond_21} and \eqref{main_cond_22}, respectively, thanks to the scaling condition \eqref{sc}. 
For \eqref{i03}, first inequality can be removed by the first one of \eqref{acad} and so the remainder implies the first one of \eqref{main_cond_1} under \eqref{acad}. 
The lower-side of the diagonal case can be checked in the similar way.
The cases $q=\infty$ and $\tilde{q}=\infty$ can be obtained by fixing $q$ and $\tilde{q}$ and perturbing the exponents $r, \tilde{r}$, but we omit the detail. Hence the proof is done.

\subsection{Time-localized estimates}

We first investigate the decay estimate in the weighted space.
The special case $p=q$ and $\gamma=\tilde{\gamma}$ was shown in \cite{OR}. 
In this work, it is required to get the decay estimates involved with different $\gamma$ and $\tilde{\gamma}$.

\begin{lem}\label{lem1}
	Let $\gamma, \tilde{\gamma} >0$.
	If $1<p'\leq q<\infty$, $ \gamma < 3/q$ and $\tilde{\gamma}-\gamma =3(1/p-1/q)$,
	then we have
	\begin{equation}\label{wd}
		\|e^{it\Delta}f\|_{L^q (|\cdot|^{-q\gamma})}\lesssim |t|^{\frac{3}{2}(\frac{1}{q}-\frac{1}{p'})-\frac{\gamma +\tilde{\gamma}}{2}}\|f\|_{L^{p'}(|\cdot|^{p'\widetilde{\gamma}})}.
	\end{equation}
\end{lem}

\begin{proof}
	We recall that the solution can be represented as
	\begin{align*}
		e^{it\Delta} f(x) 
		&=(4\pi i t)^{-\frac{3}{2}} \int e^{\frac{i|x-y|^2}{4t}} f(y) dy\\
		&=(4\pi i t)^{-\frac{3}{2}} 
		e^{\frac{ix^2}{4t}} \int e^{-i x\cdot \left(\frac{y}{2t} \right) 
			+ it \left(\frac{y}{2t}\right)^2  } f(y) dy.
	\end{align*}
	We use the change of variable $y/2t \rightarrow y$ to see
	\begin{align*}
		e^{it\Delta} f(x) 
		&=\frac{(2t)^{3}}{(4\pi i t)^{3/2}} e^{\frac{ix^2}{4t}} \mathcal{F}(e^{it|\cdot|^2} f(2t\cdot)) (x).
	\end{align*}
	We apply the Pitt-inequaltiy to above (See \cite{Pi, S3} and \cite{BH});
	\begin{equation*}
		\||\xi|^{-\gamma} \hat{f} \|_{L^q} \leq C \| |x|^{\tilde{\gamma}} f\|_{L^{p'}}
	\end{equation*}
	which holds if and only if $1< p'\leq q <\infty$, $\tilde{\gamma}-\gamma=3(1/p-1/q)$, $0\leq \gamma< 3/q$ and $0\leq \tilde{\gamma} <3/p$.
	Then 
	\begin{align*}
		\||x|^{-\gamma} e^{it\Delta} f\|_{L_x^q}
		&=\frac{(2t)^{3}}{(4\pi i t)^{3/2}} \left\| |x|^{-\gamma}  \mathcal{F}(e^{it|\cdot|^2} f(2t\cdot))\right\|_{L_x^q} \\
		&\leq \frac{C(2t)^{3}}{(4\pi i t)^{3/2}} \left\||y|^{\tilde{\gamma}} f(2t\cdot) \right\|_{L_y^{p'}}\\
		&=\frac{C(2t)^{3-3/p'-\tilde{\gamma}}}{(4\pi i t)^{3/2}} 
		\left\| |y |^{\tilde{\gamma}} f \right\|_{L_y^{p'}}\\
		& \sim t^{\frac{3}{2}-\frac{3}{p'}-\tilde{\gamma}} \left\||y|^{\tilde{\gamma}} f \right\|_{L_y^{p'}}
	\end{align*}
	by using the change of variable again.
	The proof is done.
\end{proof}

We only need to show that the local estimates
\begin{equation*}
	|T_{j, Q}(F, G)| \lesssim 2^{-j\beta( a, b, \tilde{a}, \tilde{b})} \|F\|_{L_t^{a'}(I; L_x^{b'}(|x|^{b'\gamma}))} \|G\|_{L_t^{\tilde{a}'}(J; L_x^{\tilde{b}'}(|x|^{\tilde{b}' \tilde{\gamma}}))}
\end{equation*}
hold with $\beta$ given as in \eqref{be} for the following cases: 
\begin{enumerate}
	\item[$(a)$] $\frac{1}{b}-\frac{\gamma}{3}=\frac{1}{\tilde{b}}-\frac{\tilde{\gamma}}{3}$, $\frac{\gamma}{3}<\frac{1}{b} \leq \frac{1}{2}-\frac{\tilde{\gamma}-\gamma}{6}\ \text{and}\ 1\leq \tilde{a}' \leq a\leq \infty,$ \vspace{5pt}
	\item[$(b)$] $(\frac{1}{b},\, \frac{1}{\tilde{b}})=(\frac{1}{6} +\frac{\gamma}{3},\, \frac{1}{2}+\frac{\tilde{\gamma}-\varepsilon}{3})$ and $1\leq \tilde{a}' \leq a \leq 2$, \vspace{5pt}
	\item[$(c)$] $(\frac{1}{b},\, \frac{1}{\tilde{b}})= (\frac{1}{2}+\frac{\gamma-\varepsilon}{3},\, \frac{1}{6}+\frac{\tilde{\gamma}}{3} )$ and $2\leq \tilde{a}' \leq a\leq \infty$,\vspace{5pt}
	\item[$(d)$] $(\frac{1}{b},\, \frac{1}{\tilde{b}})=(\frac{\gamma+\varepsilon}{3}, \frac{1+\tilde{\gamma}}{3})$ and $a=\tilde{a}=2$ for $\tilde{\gamma}<\frac{1}{2}$,\vspace{5pt}
	\item[$(e)$] $(\frac{1}{b},\, \frac{1}{\tilde{b}} )= (\frac{1+\gamma}{3},  \frac{\tilde{\gamma}+\varepsilon}{3})$ and $a=\tilde{a}=2$ for $\gamma<\frac{1}{2}$, \vspace{5pt}
\end{enumerate}
for sufficiently small $\varepsilon>0$.
Then the case $(a)$ implies the desired one for the diagonal case $1/b-\gamma/3=1/\tilde{b}-\tilde{\gamma}/3$.
It is not difficult to get the upper-side (resp. lower-side) of diagonal case by interpolating with cases $(a)$, $(b)$ and $(d)$ (resp. $(a)$, $(c)$ and $(e)$). We attach the details in Appendix \ref{appen_B}.

Let us start the proof of $(a)$.
By using Lemma \ref{lem1} and the H\"older inequality, we obtain 
\begin{align*}
	\left|T_{j, Q}(F, G) \right|
	& = \left| \int_J \int_I \left\langle e^{-is\Delta} F,\ e^{-it\Delta} G \right\rangle ds dt \right| \\
	&\lesssim \int_J \int_I \| |x|^{-\tilde{\gamma}} e^{i(t-s)\Delta} F \|_{L_x^{\tilde{b}}} \||x|^{\tilde{\gamma}}G \|_{L_x^{\tilde{b}'}}\ ds\, dt\\
	&\lesssim  \int_J \int_I |t-s|^{\frac{3}{2}(\frac{1}{\tilde{b}}-\frac{1}{b'})-\frac{\gamma+\tilde{\gamma}}{2}} \| |x|^{\gamma} F \|_{L_x^{b'}} \||x|^{\tilde{\gamma}}G \|_{L_x^{\tilde{b}'}}\ ds\, dt
\end{align*}
where $\gamma-\tilde{\gamma}=3(1/b-1/\tilde{b})$, $0\leq \gamma <3/b$, $\tilde{\gamma}>0$ and $1<b'\leq \tilde{b} <\infty$.
Since $|t-s|\sim 2^j$ for $s\in I$ and $t \in J$, we have
\begin{align*}
	\left|T_{j, Q}(F, G) \right|
	&\lesssim  2^{j (\frac{3}{2}(\frac{1}{\tilde{b}}-\frac{1}{b'})-\frac{\gamma+\tilde{\gamma}}{2} )} \|F\|_{L_t^{1}(I;L_x^{b'}(|x|^{b'\gamma}))}  \|G \|_{L_t^{1}(J;L_x^{\tilde{b}'}(|x|^{\tilde{b}'\tilde{\gamma}}))}\\
	& \lesssim 2^{-j\beta(a, b, \tilde{a}, \tilde{b})} \|F\|_{L_t^{a'}(I;L_x^{b'}(|x|^{b'\gamma}))}  \|G \|_{L_t^{\tilde{a}'}(J;L_x^{\tilde{b}'}(|x|^{\tilde{b}'\tilde{\gamma}}))}
\end{align*}
by the H\"older inequality with respect to $t,s$.

For case $(d)$, we see that
\begin{align*}
	\left|T_{j, Q}(F, G) \right|
	&\leq  \int_J \left\langle | \int_I e^{i(t-s)\Delta} F ds| ,\  |G(t)| \right\rangle_x  dt  \\
	&=  \left\langle |  |x|^{-\tilde{\gamma}}\int_I \chi_{J}(t)e^{i(t-s)\Delta} F ds| ,\ | |x|^{\tilde{\gamma}}\chi_{J}(t)G(t)| \right\rangle_{x,t}  . 
\end{align*}
By the H\"older inequality with respect to $x, t$ and then applying Theorem \ref{thm_homo} with $s=\frac{1}{2}-\varepsilon$ for sufficiently small $\varepsilon>0$, we see that 
\begin{align*}
	\left|T_{j, Q}(F, G) \right|
	&\leq  \left\|  \int_I e^{i(t-s)\Delta} F(s) ds \right\|_{L_t^2(\mathbb{R}; L_x^{\tilde{b}}(|x|^{-\tilde{b}\tilde{\gamma}}))} 
	\left\| G \right\|_{L_t^{2}(J;L_x^{\tilde{b}'}(|x|^{\tilde{b}'\tilde{\gamma}})) }\\
	&\lesssim  \left\| F \right\|_{L_t^{(\frac{2}{1-\varepsilon})'}(I;L_x^{b'}(|x|^{b'\gamma})) }\left\| G \right\|_{L_t^{2}(J;L_x^{\tilde{b}'}(|x|^{\tilde{b}'\tilde{\gamma}})) }
\end{align*}
where 
$$\frac{1}{b}=\frac{1+\gamma}{3}, \quad \frac{1}{\tilde{b}}=\frac{\tilde{\gamma}+\varepsilon}{3}
$$ 
for $\tilde{\gamma}<1$ and $\gamma<1/2$. 
By the H\"older inequality again, it leads us to get the desired estimate when $a= \tilde{a}= 2$, and $b, \tilde{b}$ is given as above.
By symmetry, we also obtain the case $(e)$.

It is more delicate to prove the case $(b)$, because we need to get $1/\tilde{b}>1/2$.
For this, we consider a ball $B$ in $\mathbb{R}$ whose center is a origin and radius is $2^{\frac{j}{2}}$ and then decompose the space domain of $G$ into two part, $$G(x, t)=\chi_{B}(|x|)G(x, t)+ \chi_{B^c}(|x|) G(x,t).$$

On $B$, by using H\"older's inequality and then applying \eqref{wd} to the first term, we see that
\begin{align*}
	\langle e^{-is\Delta}F, e^{-it\Delta} \chi_{B} G  \rangle
	&\leq \||x|^{-\tilde{\gamma}} e^{i(t-s)\Delta} F\|_{L_x^{\frac{6}{1+2\tilde{\gamma}}}} \| |x|^{\tilde{\gamma}} \chi_{B} G\|_{L_x^{\frac{6}{5-2\tilde{\gamma}}}}\\
	&\lesssim |t-s|^{-1}\||x|^{\gamma} F\|_{L_x^{\frac{6}{5-2\gamma}}} \| |x|^{\tilde{\gamma}} \chi_{B} G\|_{L_x^{\frac{6}{5-2\tilde{\gamma}}}}.
\end{align*}
Since $\| \chi_{B}\|_{L_x^a}\sim 2^{\frac{3j}{2a}}$ for $0\leq a\leq 1$,
we use H\"older's inequality in $x$-variable for the second term to be bounded as
\begin{equation*}
	\| |x|^{\tilde{\gamma}} \chi_{B} G\|_{L_x^{\frac{6}{5-2\tilde{\gamma}}}} \leq 2^{\frac{3j}{2}(\frac{5-2\tilde{\gamma}}{6}-\frac{1}{\tilde{b}'})} \||x|^{\tilde{\gamma}}G\|_{L_x^{\tilde{b}'}}
\end{equation*}
if $\frac{1}{\tilde{b}'} \leq \frac{5-2\tilde{\gamma}}{6}.$
From this, we make use of H\"older's inequality in $t, s$-variables to get
\begin{align*}
	\left|T_{j, Q}(F, G) \right|& = \left| \int_J \int_I \left\langle e^{-is\Delta} F,\ e^{-it\Delta} G \right\rangle ds\, dt \right| \cr
	&\leq 2^{-j+\frac{3j}{2}(\frac{5-2\tilde{\gamma}}{2n}-\frac{1}{\tilde{b}'})}\int_J \int_I  \||x|^{\gamma} F\|_{L_x^{\frac{6}{5-2\gamma}}} \||x|^{\tilde{\gamma}}G\|_{L_x^{\tilde{b}'}} ds\, dt\\
	&\leq 2^{-j\beta(a, b,\tilde{a},\tilde{b})} \|F\|_{L_t^{a'}(I;L_x^{\frac{6}{5-2\gamma}}(|x|^{\frac{6\gamma}{5-2\gamma}}))}  \|G \|_{L_t^{\tilde{a}'}(J;L_x^{\tilde{b}'}(|x|^{\tilde{b}'\tilde{\gamma}}))}
\end{align*}
for $1\leq a, \tilde{a} \leq \infty$ since $|t|\sim 2^j$.
On the other hand, for $B^c$ we will use the dual estimate of \eqref{1}.
To do so, we first put integral over $s$ into inner product on $L_x^2$ together with H\"older's inequality to see
\begin{align*}
	\left|T_{j, Q}(F, G) \right|
	&\leq  \int_J \left| \left\langle \int_I e^{-is\Delta} F ds ,\ e^{-it\Delta}\chi_{B^c}G \right\rangle_{L_x^2} \right| dt  \cr
	&\leq  \left\|  \int_I  e^{-is\Delta} F(s) ds \right\|_{L_x^2} \int_J \| e^{-it\Delta} \chi_{B^c}  G \|_{L_x^2} dt.
\end{align*}
By using the dual estimate of \eqref{1} with $s=0$, $q=2$ and $r=\frac{6}{1+2\gamma}$ the first term can be bounded as
\begin{align*}
	\left\| \int_I  e^{-is\Delta} F(s) ds \right\|_{L_x^2} \leq \||x|^{\gamma} F \|_{L_t^{2}(I;L_x^{\frac{6}{5-2\gamma}})}
\end{align*}
if $0\leq\gamma<3/b$.
Notice that we apply the Plancherel theorem along with H\"older's inequality to the second term so that
\begin{align*}
	\| e^{-it\Delta}\chi_{B^c}  G \|_{L_x^2} = \left\| \chi_{B^c}G \right\|_{L_x^{2}}
	& \leq \| |x|^{-\tilde{\gamma}}\|_{L_x^{\frac{2\tilde{b}}{2-\tilde{b}}}(B^c)} \left\| |x|^{\tilde{\gamma}} G \right\|_{L_x^{\tilde{b}'}}\cr
	& = 2^{-j\tilde{\gamma}+\frac{3j}{2}(\frac{1}{\tilde{b}}-\frac{1}{2})} \||x|^{\tilde{\gamma}} G \|_{L_x^{\tilde{b}'}}
\end{align*}
since $|x|^{-\tilde{\gamma}} \in L_x^{\frac{2\tilde{b}}{2-\tilde{b}}}(B^c)$ if $1/2< 1/\tilde{b}<1/2+\tilde{\gamma}/3$.
By H\"older's inequality with respect to $t,s$, we obtain
\begin{align*}
	\left|T_{j, Q}(F, G) \right|
	&\leq  2^{-\frac{j\tilde{\gamma}}{2}+\frac{3j}{2}(\frac{1}{\tilde{b}}-\frac{1}{2})} \||x|^{\gamma} F \|_{L_t^{2}(I;L_x^{\frac{6}{5-2\gamma}})} \int_J \||x|^{\tilde{\gamma}} G \|_{L_x^{\tilde{b}'}} dt\\
	& \leq 2^{-j\beta(a, b;\tilde{a}, \tilde{b})} \|F\|_{L_t^{a'}(I;L_x^{\frac{6}{5-2\gamma}}(|x|^{\frac{6\gamma}{5-2\gamma}}))}  \|G \|_{L_t^{\tilde{a}'}(J;L_x^{\tilde{b}'}(|x|^{\tilde{b}'\tilde{\gamma}}))}
\end{align*}
where $1 \leq a \leq 2$, $1\leq \tilde{a}\leq \infty$.
Hence, we obtain the case $(b)$ as desired.
The case $(c)$ is also obtained in the similar way.

\section{Proof of Theorem \ref{thm}}\label{sec3}
This section is devoted to establishing the well-posedness result for \eqref{3DINLS}.
Throughout this section, we use the notation $\beta=4-2\alpha$ as the energy critical index for simplicity.
We first provide some nonlinear estimates to handle the strong singularity on the nonlinear term. 
The following lemma is the inhomogeneous Strichartz estimates which immediately follows from Theorem \ref{thm_inhomo}.

\begin{lem}\label{lem_inhomo St}
 The following inhomogeneous Strichartz estimates hold;
	\begin{equation}\label{10}
		\left\|  \int_0^t e^{i(t-s)\Delta} F(s)\, ds \right\|_{W}\lesssim \| F\|_{N},
	\end{equation}
	\begin{equation}\label{11}
		\left\| \int_0^t e^{i(t-s)\Delta} F(s)\, ds \right\|_{\dot{W}^1}\lesssim \| F\|_{\dot{N}^1},
	\end{equation}
	for
	\[
	\gamma_0< \min\{\frac{1}{2}, \frac{3}{3\beta+5}\}.
	\]
	Also, we have
	\begin{equation}\label{08}
		\left\|  \int_0^t e^{i(t-s)\Delta} F(s)\, ds \right\|_{S_{\gamma}}\lesssim \| F\|_{N}, 
	\end{equation}
	\begin{equation}\label{09}
		\left\|  \int_0^t e^{i(t-s)\Delta} F(s)\, ds \right\|_{\dot{S}_{\gamma}^1}\lesssim \| F\|_{\dot{N}^1}.
	\end{equation}
\end{lem}

\begin{proof}
	We first show \eqref{10} and \eqref{11}. 
	For this, it is sufficient to check that $(q_0, r_0)$ and $(\ti{q}_0, \ti{r}_0)$ satisfy all requirements given as in Theorem \ref{thm_inhomo}. 
	Let us consider
	\[
\begin{aligned}
	&\lt( \frac{1}{q_0},\ \frac{1}{r_0}\rt) = \lt( \frac{3}{2(3\beta+5)}+\frac{3\delta}{2},\ \frac{1}{2}-\frac{1}{3\beta+5}+\frac{\gamma_0}{3}-\delta\rt),\\
	\lt( \frac{1}{\ti{q}_0},\ \frac{1}{\ti{r}_0}\rt)& = \lt( \frac{1}{2}+\frac{1}{3\beta+5}-\frac{3\delta(\beta+1)}{2},\ \frac{1}{6}-\frac{2}{3(3\beta+5)}+\frac{\tilde{\gamma}_0}{3}+\delta(\beta+1)\rt),
\end{aligned}
\]
with $\beta=4-2\alpha$.	
Then we have $(q_0, r_0)\in \mathcal{AD}_{\gamma_0} \subset \mathcal{AC}_{\gamma_0}$ for $$\delta\leq \frac{1}{3}-\frac{1}{3\beta+5} \qquad \text{and}\qquad \gamma_0\leq \frac{3}{3\beta+5},$$ and 
$(\ti{q}_0, \ti{r}_0)\in \mathcal{AC}_{\ti{\gamma}_0}$ for $$\delta(\beta+1)<\frac{1}{3}+\frac{2}{3(3\beta+5)}.$$	
It is not difficult to see that \eqref{main_cond_1} is true. 
Thus these all conditions above are true by taking $\delta>0$ sufficiently small.
On the other hand, for small $\delta>0$ we observe that 
\bq\label{cond_4}
\frac{1}{\ti{r}_0}<\frac{1}{6}+\frac{\ti{\gamma}_0}{3}
\eq
and so $(1/r_0, 1/\tilde{r}_0)$ lies in the triangle $B'C'D'$ in Figure \ref{pic}.
Thus we get  
\[
\frac{1}{r_0}-\frac{\gamma_0}{3}>\frac{1}{\ti{r}_0}-\frac{\ti{\gamma}_0}{3}
\]
and so we only need to check the conditions \eqref{main_cond_31} and \eqref{main_cond_32}.
For \eqref{main_cond_31}, it is trivial, see Figure \ref{pic}.
Also, \eqref{main_cond_32} is valid with $\gamma_0<1/2$.
	
	Next we prove \eqref{08} and \eqref{09}.
	We note that $(q, r)\in \mathcal{AD}_{\gamma}$ and $(\tilde{q}_0, \tilde{r}_0)\in \mathcal{AC}_{\tilde{\gamma}_0}$ for sufficient small $\delta>0$.
	The first and second ones of \eqref{main_cond_1} is satisfied if
	\[
	\frac{1}{q}<\frac{1}{2}-\frac{1}{3\beta+5}+\frac{3\delta(\beta+1)}{2}, \qquad \frac{1}{r}<\frac{1}{2}-\frac{2}{3(3\beta+5)}+\frac{\gamma}{3}+\delta(\beta+1)
	\]
	thanks to \eqref{cond_4}.
	By the scaling condition, it follows that
	\bq\label{con1}
	\frac{1}{6}+\frac{\gamma}{3}+\frac{2}{3(3\beta+5)}-\delta(\beta+1)<\frac{1}{r} <\frac{1}{2}+\frac{\gamma}{3}-\frac{2}{3(3\beta+5)}+\delta(\beta+1).
	\eq
	Due to \eqref{cond_4}, the condition \eqref{main_cond_31} is automatically true. 
	For \eqref{main_cond_32} to hold, it is required that 
	\bq\label{con2}
	\frac{1}{r}<\frac{1}{2}+\frac{\gamma}{3}-\frac{1}{3\beta+5}+\frac{3\delta(\beta+1)}{2}\quad \text{and} \quad \gamma<\frac{1}{2}.
	\eq
	Since the upper bound of \eqref{con1} can be removed by \eqref{con2}, 
	thus, \eqref{con1} and \eqref{con2} are reduced to \eqref{cond_qr} with $\gamma<1/2$
	and sufficiently small $\delta$.
	Hence, the proof is done.
\end{proof}

Next, we obtain some estimates of the nonlinear term $F(u)=|x|^{-\alpha} |u|^{\beta}u$ where $\beta=4-2\alpha$ for $\alpha\ge 3/2$.

\begin{lem}\label{lem_nonlinear}
	Let $3/2\leq \alpha<11/6$.
	Assume that
	\[
	\begin{aligned}
		-\gamma_0 -1 \leq \frac{\tilde{\gamma}_0-\alpha+\gamma_0}{\beta} \leq -\gamma_0.\\
	\end{aligned}
	\]
	Then the following estimates hold;
	\begin{equation}\label{3}
		\|F(x, u)\|_{\dot{N}^1} \lesssim \|u\|_{\dot{W}^{1}}^{\beta+1}, 
	\end{equation}
	\begin{equation}\label{4}
		\|F(x, u)\|_{N} \lesssim \|u\|_{\dot{W}^{1}}^{\beta} \|u\|_{W},
	\end{equation}
	\begin{equation*}\label{5}
	\|F(x, u)\|_{\dot{N}^1} \lesssim \|u\|_{\dot{W}^{1}}^{\beta}  \|u\|_{\dot{S}_{\gamma}^{1}},
\end{equation*}
\begin{equation*}\label{6}
	\|F(x, u)\|_{N} \lesssim \|u\|_{\dot{W}^{1}}^{\beta} \|u\|_{S_{\gamma}}.
\end{equation*}
\end{lem}

\begin{proof}
Thanks to \eqref{relation_WS}, we only need to show \eqref{3} and \eqref{4}. To show \eqref{3}, one has
	\[
	\begin{aligned}
		\||x|^{\ti{\gamma}_0} \nabla (|x|^{-\alpha}|u|^{\beta}u)\|_{L_t^{\ti{q}_0'}L_x^{\ti{r}_0'}} 
		&\lesssim \||x|^{\ti{\gamma}_0-\alpha-1} |u|^{\beta}u\|_{L_t^{\ti{q}_0'} L_x^{\ti{r}_0'}}\\ 
		&\quad +\||x|^{\ti{\gamma}_0-\alpha} |u|^{\beta-1}u\nabla u\|_{L_t^{\ti{q}_0'} L_x^{\ti{r}_0'}}\\
		&\quad +\||x|^{\ti{\gamma}_0-\alpha} |u|^{\beta}\nabla u\|_{L_t^{\ti{q}_0'} L_x^{\ti{r}_0'}}\\
		& =: I+II+III
	\end{aligned}
	\]
	from the direct calculation.
	By applying the weighted Sobolev embedding (see e.g. \cite{SW});
	\begin{equation}\label{embedding}
		\big\||x|^{b}f\big\|_{L^q} \lesssim \big\||x|^{a} \nabla f\big\|_{L^p}
	\end{equation}
	for $1\leq p\leq q<\infty$, $-3/q<b\leq a$ and $a-b-1=3/q-3/p$, 
	it is easy to see 
	\[
	\begin{aligned}	
		II, III &\lesssim \||x|^{\frac{\ti{\gamma}_0-\alpha+\gamma_0}{\beta}} u\|_{L_t^{q_0}L_x^{p}}^{\beta} \||x|^{-\gamma_0} \nabla u\|_{L_t^{q_{0}}L_x^{r_{0}}}\\
		&\lesssim \||x|^{-\gamma_0} \nabla u\|_{L_t^{q_0}L_x^{r_0}}^{\beta+1} \\
		\end{aligned}
	\]
	with
	\[
	\begin{aligned}
		&\frac{1}{p}= \frac{1}{r_0}-\frac{\gamma_0}{3}-\frac{1}{3}-\frac{\ti{\gamma}_0-\alpha+\gamma_0}{3\beta}
	\end{aligned}
	\]
	if 
	\[
	-\gamma_0 -1+\sigma \leq \frac{\tilde{\gamma}_0-\alpha+\gamma_0}{\beta} \leq -\gamma_0, \qquad \frac{\gamma_0+1}{3}<\frac{1}{r_0}
	\]
where the second condition above is valid under $1/3<\beta$, equivalent to $\alpha<11/6$.
	Also, from \eqref{embedding} the term $I$ can be estimated as
	\[
	\begin{aligned}
	I	
	&\lesssim \||x|^{\frac{\ti{\gamma}_0-\alpha+\gamma_0}{\beta}} u\|_{L_t^{q_{0}}L_x^{p}}^{\beta} \||x|^{-\gamma_{0}-1} u\|_{L_t^{q_{0}}L_x^{r_{0}}}\\
	&\lesssim \||x|^{-\gamma_{0}} \nabla u\|_{L_t^{q_0}L_x^{r_0}}^{\beta+1}  
	\end{aligned}
	\]
	under the same condition to get $II$ and $III$.
	The inequality \eqref{4} is obviously obtained in the quite similar way to get \eqref{3}.
	We omit the detail and so the proof is done.
\end{proof}

Now we start with the proof of Theorem \ref{thm} via an picard iteration argument which gives useful information for the perturbation theory.
Let us consider a sequence $\{u^{(m)}\}_{m=0}^{\infty}$ with $u^{(0)}(t):=0$ satisfying
\begin{equation*}
	\left\{
	\begin{aligned}
		&i \partial_{t} u^{(m)} + \Delta u^{(m)} = \lambda F(x, u^{(m-1)}), \\
		&u^{(m)}(x, 0)=u_0
	\end{aligned}
	\right.
\end{equation*}
for  $m \ge1$.
By Duhamel's principle, it is written as 
\begin{equation*}
	u^{(m)}=e^{it\Delta} u_0 -i\int_0^t e^{i(t-s)\Delta} F(x, u^{(m-1)}) ds.
\end{equation*}
We first show $\{u^{(m)}\}_{m=0}^{\infty}$ is cauchy sequence in $S_{\gamma}$.
Using \eqref{11} and Lemma \ref{lem_nonlinear}, we see that
\begin{align*}
	\|u^{(m)}\|_{\dot{W}^{1}}
	&\lesssim \|e^{it\Delta} u_0\|_{\dot{W}^{1}}+\|F(x, u^{(m-1)})\|_{\dot{N}^{1}}\\
	&\lesssim \eta+\|u^{(m-1)}\|_{\dot{W}^{1}}^{\beta+1}
\end{align*}
by \eqref{as}.
Then by an induction argument, there is $\eta<1$ sufficiently small so that
\bq\label{sm}
	\|u^{(m)}\|_{\dot{W}^{1}} \lesssim \eta, \qquad \text{for}\quad m\ge0
\eq
We apply \eqref{08}, \eqref{4} and \eqref{sm} in turn. Then it follows that 
\begin{align*}
	\|u^{(m)}\|_{S_{\gamma}}
	&\lesssim \|u_0\|_{L^2}+\|F(x, u^{(m-1)})\|_N \nonumber\\
	&\lesssim \|u_0\|_{L^2}+\|u^{(m-1)} \|_{\dot{W}^{1}}^{\beta} \|u^{(m-1)}\|_{S_{\gamma}} \nonumber\\
	&\lesssim \|u_0\|_{L^2}+\eta^{\beta} \|u^{(m-1)}\|_{S_{\gamma}}.
\end{align*}
Thus, we deduce from the induction argument that
\begin{equation*}
	\|u^{(m)}\|_{S_{\gamma}} \lesssim \|u_0\|_{L^2},
\end{equation*}
for sufficiently small $\eta<1$. Thus we have $u^{(m)} \in S_{\gamma}$ for $m\ge0$.

On the other hands, by \eqref{08}, \eqref{4} and \eqref{sm} one has
\begin{align}
	\|u^{(m+1)}-u^{(m)}\|_{S_{\gamma}}
	&\lesssim \| (F(x, u^{(m)})-F(x, u^{(m-1)}))\|_{N}\nonumber\\
	&\lesssim (\|u^{(m)}\|_{\dot{W}^{1}}^{\beta}+\|u^{(m-1)}\|_{\dot{W}^{1}}^{\beta}) \|u^{(m)}-u^{(m-1)}\|_{S_{\gamma}} \nonumber\\
	& \lesssim \eta^{\beta} \|u^{(m)}-u^{(m-1)}\|_{S_{\gamma}}. \label{sm3}
\end{align}
Then the small constant $\eta<1$ grantees that the sequence $\{u^{(m)}\}_{m=0}^{\infty}$ is the Cauchy in $S_{\gamma}$. Since $S_{\gamma}$ is a complete space, there exist a solution $u \in S_{\gamma}$ such that $u^{(m)}$ strongly converge to $u$ in $S_{\gamma}$ and so we obtain \eqref{w3}.
Thanks to \eqref{sm}, there is a subsequence $u^{(m_k)}$ weakly convergent to $u$ in  $\dot{W}^1$ and
$$ \|u\|_{\dot{W}^{1}} \leq \liminf_{k\rightarrow \infty} \|u^{(m_k)}\|_{\dot{W}^{1}} \lesssim \eta.$$
It leads us to get \eqref{w1}. 

Meanwhile,
by Lemma \ref{lem_inhomo St}, Lemma \ref{lem_nonlinear}, \eqref{w1} and \eqref{w3} we obtain
\begin{align*}\label{h2}
	\|u\|_{H^1}
	&\lesssim \|u_0\|_{H^1}+ \|F(x, u)\|_{N}+\|F(x,  u)\|_{\dot{N}^1}\\
	&\lesssim \|u_0\|_{H^1}+ \|u\|_{\dot{W}^{1}}^{\beta}\|u\|_{S_{\gamma}}+\|u\|_{\dot{W}^{1}}^{\beta}\|u\|_{\dot{W}^1} \\
	&\lesssim \|u_0\|_{H^1}+ \eta^{\beta}\|u_0\|_{L^2}+\eta^{\beta+1} <\infty
\end{align*}
and similarly 
\begin{equation*}
	\|u\|_{\dot{S}_{\gamma}^{1}} \lesssim \|u_0\|_{\dot{H}^{1}}+\eta^{\beta+1}.
\end{equation*}
Therefore, we conclude that $u \in C_t H^1 \cap \dot{S}^1_{\gamma}$ satisfying \eqref{w2}.

It remains to check the all restrictions arising from Lemma \ref{lem_inhomo St} and Lemma \ref{lem_nonlinear}.
From Lemma \ref{lem_inhomo St}, one has the restrictions 
\begin{align}
	&\frac{3}{2}\leq \alpha<\frac{11}{6},\qquad 0<\gamma_0<\min\{ \frac{1}{2}, \frac{3}{3\beta+5}\}, \qquad 0<\ti{\gamma}_0<1 \label{c1}
\end{align}
and from Lemma \ref{lem_nonlinear} we also need the condition
\begin{align}
&-\gamma_0 (\beta+1) -\beta +\alpha\leq \ti{\gamma}_0 \leq -\gamma_0 (\beta+1)+\alpha. \label{c5}
\end{align}
Then there exist $\ti{\gamma}_0$ whenever
\bq\label{c6}
-1+\frac{\alpha}{\beta+1}<\gamma_0 <\frac{\alpha}{\beta+1}.
\eq
by comparing the last condition of \eqref{c1} and \eqref{c5}.
Similarly, to exist $\gamma_0$ we make the lower bounds of $\gamma_1$ in the second condition of \eqref{c1} and \eqref{c6} less than the upper ones and so it follows that
\[
\begin{aligned}
	\alpha<\frac{15}{8}, \qquad -(3\beta-2)(3\beta+5)<6(\beta+1).
\end{aligned}
\]
We note here that these two conditions above are redundant for $3/2\leq \alpha<11/6$ due to $\beta=4-2\alpha$.

Finally, we prove the continuous dependence \eqref{co}. Let us consider $u$ and $\tilde{u}$ is a solution with respect to a initial data $u_0$ and $\tilde{u}_0$, respectively, satisfying \eqref{as}.
In the similar way to show \eqref{sm3}, it is easy to show
\begin{align*}
	\|u-\tilde{u}\|_{S_{\gamma}}
	&\lesssim \|u_0 -\tilde{u}_0\|_{L^2} +\|F(x, u)-F(x, \tilde{u})\|_{N}\\
	&\lesssim \|u_0-\tilde{u}_0 \|_{L^2} +(\|u\|^{\beta}_{\dot{X}^{1-\sigma}}+\|\tilde{u}\|^{\beta}_{\dot{X}^{1-\sigma}})\|u-\tilde{u}\|_{S_{\gamma}}\\
	&\lesssim \|u_0-\tilde{u}_0 \|_{L^2} + \eta^{\beta} \|u-v\|_{S_{\gamma}}.
\end{align*}
Since $\eta<1$ is small enough, we conclude that \eqref{co}.

\section{Necessary conditions}\label{sec4}
In this section, we discuss the sharpness of Theorem \ref{thm}.
Following \cite{V2, F}, 
we investigate some necessary conditions for the weighted inhomogeneous estimate \eqref{inhomo}.
\begin{prop}\label{prop}
Let $n\ge 3$.
Suppose that \eqref{inhomo} holds,
	then $q, r, \tilde{q}, \tilde{r}$ must satisfy the following conditions 
	\begin{equation}\label{n0}
		\frac{1}{q}+\frac{1}{\tilde{q}}=\frac{n}{2}(1-\frac{1}{r}-\frac{1}{\tilde{r}})+\frac{\gamma+\tilde{\gamma}}{2},
	\end{equation}
	\begin{equation}\label{n1}
		\frac{\gamma}{n}<\frac{1}{r} < \frac{1}{2}+\frac{\gamma}{n}, \quad \frac{\tilde{\gamma}}{n}<\frac{1}{\tilde{r}}<\frac{1}{2}+\frac{\tilde{\gamma}}{n},
		\quad \frac{1}{q}+\frac{1}{\tilde{q}}\leq 1,
	\end{equation}
	\begin{equation}\label{n2}
		\frac{1}{q}< n(\frac{1}{2}-\frac{1}{r})+\gamma, \quad \frac{1}{\tilde{q}}<n(\frac{1}{2}-\frac{1}{r})+\tilde{\gamma}.
	\end{equation}
	\begin{equation}\label{n5}
		\left| (\frac{1}{r}-\frac{\gamma}{n})-(\frac{1}{\tilde{r}}-\frac{\tilde{\gamma}}{n})\right| \leq \frac{1}{n}.
	\end{equation}
\end{prop}

The first necessity of \eqref{n0} follows from the scaling argument under $(t, x) \rightarrow (\delta^2 t, \delta x)$ for $\delta>0$.
For the last one of \eqref{n1},
we note that $TT^{\ast}$ has a convolution structure with respect to $t$, in fact
\begin{align}\label{kernel}
	TT^{\ast}F 
	&=(4\pi)^{-\frac{n}{2}} \int_0^t \int_{\mathbb{R}^n} |t-s|^{-\frac{n}{2}} e^{ \frac{i|x-y|^2}{4|t-s|}} F(y, s) dy ds.
\end{align}
Since it is a time translation invariant operator, the requirement $q>\ti{q}'$ follows from Theorem 1 in \cite{H}.
We note that the operator $TT^{\ast}$ also has convolution structure with respect to $x$ though, it is not space-translation invariant in the weighted space anymore.

\subsection{Proof of \eqref{n2}}
Let $\phi$ be a smooth cut-off function supported in the interval $[1,2]$.
We consider
$$\widehat{F(\cdot, s)}(\xi) =\chi_{\{0<s<1 \}} \phi(|\xi|)$$
and then
 $\|F\|_{L_s^{\tilde{q}'}L_y^{\tilde{r}'}(|y|^{\tilde{r}'\tilde{\gamma}})} \lesssim 1$,
because $F \in \mathcal{S}$ implies $|y|^{\tilde{r}'\tilde{\gamma}}F \in \mathcal{S}$.

We first see that
\begin{align*}
	TT^{\ast}F
	&= (2\pi)^{-\frac{n}{2}}\int_0^1 \int_{\mathbb{R}^n} e^{i x \cdot \xi + i(t-s)|\xi|^2} \phi(|\xi|) d\xi ds \nonumber\\
	&= (2\pi)^{-\frac{n}{2}}\int_{\mathbb{R}^n} e^{i x \cdot \xi + it|\xi|^2} \left(\int_0^1 e^{-is|\xi|^2} ds \right) \phi(|\xi|) d\xi \nonumber\\
	&= (2\pi)^{-\frac{n}{2}}\int_{\mathbb{R}^n} e^{i x \cdot \xi + it|\xi|^2} \left( \frac{e^{-i|\xi|^2}-1}{-i|\xi|^2} \right) \phi(|\xi|) d\xi
\end{align*}
Now, we shall use the stationary phase method (p. 344, \cite{S});
\begin{equation}\label{stationary_1}
	I(M) := \int e^{iMq(\xi)} \varphi(\xi) d\xi \sim M^{-\frac{n}{2}} \qquad \text{for sufficiently large } M
\end{equation}
where a phase function $q$ has a nondegenerate critical point at $\xi_0 \in supp(\varphi)$,
i.e. $\nabla q(\xi_0)=0$ and the matrix $[\frac{\partial^2 q}{\partial \xi_i \partial \xi_j}](\xi_0)$ is invertible.
Indeed, we will apply \eqref{stationary_1} with $\varphi(\xi) = \left( \frac{e^{-i|\xi|^2}-1}{-i|\xi|^2} \right) \phi(|\xi|)$ and $ q(\xi)=\frac{x\cdot \xi}{t}+|\xi|^2$ when $t\in [N, \infty)$ for sufficiently large $N>0$.
We note here that
$$ \nabla q(\xi_0) =0  \quad \text{at} \quad \xi_0=-\frac{x}{2t} \quad \text{and}\quad \lt[\frac{\partial^2 q}{\partial \xi_i \partial \xi_j}\rt](\xi_0)\neq 0.$$
For $|x|\in (2t, 4t)$, one has $\xi_0=-x/2t \in supp(\varphi)$.
Thus by \eqref{stationary_1}, we deduce that $TT^{\ast}F (x, t) \sim t^{-n/2}$ for $t\in [N, \infty)$ and $|x|\in (2t, 4t)$ for large $N$ and so
\bq\label{note_3}
\begin{aligned}
	\left\| TT^{\ast}F  \right\|_{L_t^q L_x^r (|x|^{-r\gamma})}
	&\gtrsim \left( \int^{\infty}_{N} t^{-\frac{nq}{2}}\left( \int_{2t<|x|<4t} |x|^{-r\gamma } dx \right)^{\frac{q}{r}} dt \right)^{\frac{1}{q}} \\
	&\gtrsim \left( \int_N ^{\infty} t^{-\frac{nq}{2}}\left( \int_{2t}^{4t}  \rho^{-r\gamma +n-1} d\rho \right)^{\frac{q}{r}} dt \right)^{\frac{1}{q}}\\
	&\gtrsim \left( \int_N ^{\infty} t^{-\frac{nq}{2}+\frac{nq}{r}-q\gamma}  dt \right)^{\frac{1}{q}}
\end{aligned}
\eq 
for $\gamma/n <1/r$. 
Hence \eqref{inhomo} leads us to get 
$$\left( \int_N ^{\infty} t^{-\frac{nq}{2}+\frac{nq}{r}- q\gamma} dt \right)^{\frac{1}{q}} \lesssim 1$$
for large $N$.
To be valid, 
$-\frac{n}{2}+\frac{n}{r}-\gamma+\frac{1}{q}<0$ is required as a necessary one.
It is the first condition of \eqref{n2}.
The second one of \eqref{n2} is obtained in the similar way with replacing the exponents $q, r, \gamma$ by $\tilde{q}, \tilde{r}, \tilde{\gamma}$.

\begin{rem}
	The proof of \eqref{n2} also give counterexamples for 
	\bq\label{note_4}
	\frac{\gamma}{n}<\frac{1}{r}<\frac{1}{2}+\frac{\gamma}{n} \quad \text{and} \quad \frac{\tilde{\gamma}}{n}<\frac{1}{\tilde{r}}<\frac{1}{2}+\frac{\tilde{\gamma}}{n}
	\eq
	given as the first two conditions in \eqref{n1}. In fact, if $1/r \leq \gamma/n$, then \eqref{inhomo} fails from the second line of \eqref{note_3}. Also, if $1/r\ge 1/2+\gamma/n$, then it is false from the third line of \eqref{note_3}. 
	In the similar way, one easily check that the second one of \eqref{note_4} is required. 
\end{rem}

\subsection{Proof of \eqref{n5}}
Let $R \gg 1$ and let us consider
$$ F(s, y) = e^{2iR^2 s^2} \chi_{\{ 0<s<1,\ |y|\leq \frac{1}{2R} \}}.$$
Then we have $\|F\|_{L_s^{\tilde{q}'}L_y^{\tilde{r}'}(|y|^{\tilde{r}'\tilde{\gamma}})} = (2R)^{-n/\tilde{r}'-\tilde{\gamma}}$.
With $z=(x-y)/2R$, the formula \eqref{kernel} can be rewritten as
\begin{align*}
	TT^{\ast}F
	&= (2\pi)^{-\frac{n}{2}}  \int_{|y|\leq \frac{1}{2R}} \int_0^1 |t-s|^{-\frac{n}{2}} e^{i(2R^2 s^2+ \frac{|x-y|^2}{4|t-s|})} ds dy\\
	&= (2\pi)^{-\frac{n}{2}}  \int_{|y|\leq \frac{1}{2R}} \int_0^1 |t-s|^{-\frac{n}{2}} e^{iR^2(2 s^2+ \frac{z^2}{|t-s|})} ds dy\\
	&= (2\pi)^{-\frac{n}{2}} \int_{|y|\leq \frac{1}{2R}}  I(R; t, z)  dy
\end{align*}
where the oscillatory integral
$$I(R; t, z)=\int_0^1  e^{iR^2 q(s;t,z)} \psi(s;t) ds$$
with an amplitude function $\psi(s;t)=|t-s|^{-n/2}$ and a phase function $q(s;t, z)=2s^2 + |z|^2 /(t-s)$.

We shall use the following stationary phase method.
\begin{lem}[\cite{S2}, chapter VIII]
	Let us consider the oscillatory integral
	$$I(N)=\int_a^b e^{iN q(s)} \psi(s) ds$$
	with phase function $q \in C^5 [a, b]$
	and amplitute $ \psi \in C^2 [a, b]$.
	We assume that $q''(s) \ge 1$ for all $s \in [a, b]$ and that $q'(s_{\ast})=0$ for a point $s_{\ast} \in [a+\delta, b-\delta]$ with $\delta>0$.
	Then
	\bq\label{stationary_2}
	I(N)=\frac{J_{\ast} e^{iNq(s_\ast)}}{\sqrt{N}}+O(\frac{1}{N}) 
	\eq
	where $$J_{\ast} =e^{\frac{i\pi}{4}} \psi(s_{\ast}) \sqrt{\frac{2\pi}{q''(s_\ast)}}.$$
\end{lem}

Namely, for $q(s;t, z)=2s^2 + |z|^2 /(t-s)$, we see that
$$\partial_s q =4s-\frac{|z|^2}{(t-s)^2},\qquad \partial_s^2 q =4-\frac{2|z|^2}{(t-s)^3}.$$
For $t\in [4, 5]$ and $1 \leq |z| \leq 2$, one has $\partial_s^M q(s) \in L^{\infty}([0,1])$ for $M>1$.
Also, $q$ has exactly one non-degenerate critical point in $[0, 1]$ such as $$s_{\ast}= \frac{|z|^2}{4(t-s_{\ast})^2} \in [\frac{1}{100}, \frac{1}{9}]$$
and   
$$q'(s_{\ast})=0, \qquad q''(s_{\ast}) =4-\frac{8s_{\ast}}{t-s_{\ast}} >1.$$
Thus by \eqref{stationary_2} we deduce that
\bq\label{note_5}
I(R; t, z) =\frac{J_{\ast}(t, z) e^{iR^2 q(s_{\ast};t, z)} }{R} +O(\frac{1}{R^2})
\eq 
where
$$J_{\ast}(t,z)= e^{\frac{i\pi}{4}} \psi(s_{\ast};t) \sqrt{\frac{2\pi}{\partial_s^2 q(s_{\ast};t,z)}}.$$
Notice here that
\begin{align}\label{note_6}
	|J_{\ast}(t, z)|
	= \left|\, \frac{e^{\frac{i\pi}{4}} } {(t-s_{\ast})^{n/2}}  \sqrt{\frac{2\pi(t-s_{\ast})}{4t-12s_{\ast}}}\, \right|
	= \left|  \sqrt{\frac{2\pi (t-s_{\ast})^{1-n}}{4t -12 s_{\ast}}}\, \right| \sim 1
\end{align}
because of $|t-s_{\ast}|\sim 2$ for $s_{\ast} \in [1/100, 1/9]$ and $t \in [4,5]$.

One has
\bq\label{note_7}
e^{iR^2 q(s_{\ast}; t, z)}= e^{iR^2 q(s_{\ast}; t, x)+i(\frac{|y|^2-2x\cdot y}{4(t-s_{\ast})})}
\eq 
where  
$$ \left|\frac{|y|^2-2 x\cdot y}{4(t-s_{\ast})}\right| 
\leq \frac{|y|^2+2|x||y|}{16} 
\leq \frac{1}{16}\left(\frac{1}{(2R)^2}+\frac{1}{R}(4R+\frac{1}{2R})\right)
\leq \frac{19}{64} $$
for $2R+\frac{1}{2R} <|x|<4R -\frac{1}{2R}$, $|y|<\frac{1}{2R}$ and $R>1$.
Then it follows from \eqref{note_5}, \eqref{note_6} and \eqref{note_7} that
\begin{align*}
	|TT^{\ast}F| 
	&\gtrsim \left| \int_{|y|<\frac{1}{2R}} \frac{J_{\ast}(t, z) e^{iR^2 \varphi(s_{\ast}; t, z)}}{R} dy \right|\\
	&\sim \frac{1}{R}  \left| \int_{|y|<\frac{1}{2R}}  e^{ i(\frac{|y|^2-2x\cdot y}{4(t-s_{\ast})})  } dy \right|\\	
	&\gtrsim \frac{1}{R}  \left| \int_{|y|<\frac{1}{2R}} \cos \lt(\frac{|y|^2-2x\cdot y}{4(t-s_{\ast})}\rt)   dy \right|\\	
	&\gtrsim \frac{\cos(\frac{19}{64})}{R}   \int_{|y|<\frac{1}{2R}}  dy
	\sim \frac{1}{2^n R^{n+1}}.
\end{align*}
Hence we obtain 
\begin{align*}
	\|TT^{\ast}F\|_{L_t^{q} L_x^{r}(|x|^{-r\gamma})}
	&\gtrsim \frac{1}{2^n R^{n+1}} \left( \int_4^5 \left( \int_{2R+\frac{1}{2R} <|x|<4R -\frac{1}{2R}} |x|^{-r\gamma}  dx \right)^{\frac{q}{r}} dt \right)^{\frac{1}{q}}\\
	&\gtrsim  \frac{1}{2^n R^{n+1}} (2R)^{\frac{n}{r}-\gamma}\\ 
	&\gtrsim R^{-n-1+\frac{n}{r}-\gamma} 
\end{align*}
for $n-r\gamma>0$ and $R\gg1$.
As a result, \eqref{inhomo} leads us to get
$$ R^{-n-1+\frac{n}{r}-\gamma} \lesssim R^{-\frac{n}{\tilde{r}'}-\tilde{\gamma}}, \qquad R \gg1$$
which implies $$ (\frac{1}{r} -\frac{\gamma}{n})-(\frac{1}{\tilde{r}}-\frac{\tilde{\gamma}}{n}) \leq \frac{1}{n}.$$
In the similar way, it is clear to deduce the condition 
$(\frac{1}{\tilde{r}} -\frac{\tilde{\gamma}}{n})-(\frac{1}{r}-\frac{\gamma}{n}) \leq \frac{1}{n}$ by replacing $q, r, \gamma$ with $\tilde{q}, \tilde{r}, \tilde{\gamma}$.
Therefore, we conclude that \eqref{n5} is needed for \eqref{inhomo} to hold.

\appendix
\section{Interpolation}\label{appen_B}
Let us assume these cases for a moment.
We check that the conditions stated in Proposition \ref{lem2} come out of the interpolation argument.
By interpolating between $(a)$ and $(b)$, one has \eqref{eq2} for
\begin{gather}\label{i}
	\frac{1}{b}= \frac{1-\theta}{b_1} +\frac{\theta}{6}+\frac{\gamma\theta}{3}, \qquad
	\frac{1}{\tilde{b}}=\frac{1-\theta}{\tilde{b}_1}+\frac{\theta}{2}+\frac{\tilde{\gamma}\theta}{3}-\frac{\varepsilon\theta}{3},\\
	 \frac{1}{a}=\frac{1-\theta}{a_1}+\frac{\theta}{a_2},\qquad  \frac{1}{\tilde{a}}=\frac{1-\theta}{\tilde{a}_1}+\frac{\theta}{\tilde{a}_2}\label{ii}
\end{gather} 
under
\begin{gather}\label{in}
\frac{1}{b_1}-\frac{\gamma}{3}=\frac{1}{\tilde{b}_1}-\frac{\tilde{\gamma}}{3}, \qquad
\frac{\gamma}{3}<\frac{1}{b_1} \leq \frac{1}{2}-\frac{\tilde{\gamma}-\gamma}{6},\\
0\leq \frac{1}{a_1} \leq \frac{1}{\tilde{a}_1'}\leq 1, \qquad \frac{1}{2}\leq \frac{1}{a_2} \leq \frac{1}{\tilde{a}_2'} \leq 1,\qquad 0 \leq \theta \leq 1 \label{inn}
\end{gather}
for $0<\gamma, \tilde{\gamma}<1$.
By \eqref{i}, the condition \eqref{in} becomes
\bq\label{theta}
\frac{\theta(1-\varepsilon)}{3} =\frac{1}{\tilde{b}}-\frac{\tilde{\gamma}}{3}-\frac{1}{b}+\frac{\gamma}{3}
\eq
and 
\[
\frac{\theta}{6}<\frac{1}{b}-\frac{\gamma}{3}\leq \frac{1}{2}-\frac{\theta}{3} -\frac{(\tilde{\gamma}+\gamma)(1-\theta)}{6}.
\] 
Using \eqref{theta} to remove $\theta$, then it implies
\[
\frac{1}{\tilde{b}}-\frac{\tilde{\gamma}}{3} <3\lt( \frac{1}{b}-\frac{\gamma}{3}\rt), \qquad 
\frac{1}{b}-\frac{\gamma}{3} \leq -\frac{2-\gamma-\tilde{\gamma}}{ \gamma+\tilde{\gamma}}\lt( \frac{1}{\tilde{b}}-\frac{\tilde{\gamma}}{3}\rt)+\frac{3-\gamma-\tilde{\gamma}}{3 (\gamma+\tilde{\gamma})}.
\]
Meanwhile, we multiply the first condition of \eqref{inn} by $1-\theta$ and use \eqref{ii}. Then it follows that
\bq\label{c2-inter}
\frac{1}{a}+\frac{1}{\tilde{a}}-1+\theta-\frac{\theta}{\tilde{a}_2}\leq \frac{\theta}{a_2}\leq \frac{1}{a}, \quad \frac{\theta}{\tilde{a}_2}\leq \frac{1}{\tilde{a}}.
\eq
The second one of \eqref{ii} implies
\bq\label{c3-inter}
\frac{\theta}{2}\leq \frac{\theta}{a_2} \leq \theta-\frac{\theta}{\tilde{a}_2}, \quad 0<\frac{\theta}{\tilde{a}_2}.
\eq
We compare each lower bound of $\theta/a_2$ above to its upper bounds, we deduce that there always exist $a_2$ if
\bq\label{c4-inter}
-1+\theta+\frac{1}{\tilde{a}}\leq \frac{\theta}{\tilde{a}_2}\leq \frac{\theta}{2}, \quad \frac{1}{a}\leq \frac{1}{\tilde{a}'}, \quad \frac{\theta}{2}\leq \frac{1}{a}
\eq
in which the last condition means
\[
\frac{3}{2}\lt(\frac{1}{\tilde{b}}-\frac{\tilde{\gamma}}{3}-\frac{1}{b}+\frac{\gamma}{3}\rt) <\frac{1}{a}.
\]
Then comparing the conditions of $\theta/\tilde{a}_2$ in the last ones of \eqref{c2-inter} and \eqref{c3-inter} with the first one of \eqref{c4-inter}, we have
\[
0\leq \frac{1}{\tilde{a}} \leq 1+\frac{3\theta}{2}, \quad 0\leq \theta\leq 1.
\] 
where we note that the upper bound of $1/\tilde{a}$ is trivial because of the last two in \eqref{c4-inter}. 
Here, $0\leq \theta\leq1$ is represented as 
\[
\frac{1}{b}-\frac{\gamma}{3} \leq \frac{1}{\tilde{b}}-\frac{\tilde{\gamma}}{3}
 <\frac{1}{b}-\frac{\gamma}{3}+\frac{1}{3}
\]
by \eqref{theta}.
To conclude, one has
\begin{gather*}
	\frac{1}{3}\lt( \frac{1}{\tilde{b}}-\frac{\tilde{\gamma}}{3}\rt) < \frac{1}{b}-\frac{\gamma}{3} \leq -\frac{2-\gamma-\tilde{\gamma}}{ \gamma+\tilde{\gamma}}\lt( \frac{1}{\tilde{b}}-\frac{\tilde{\gamma}}{3}\rt)+\frac{3-\gamma-\tilde{\gamma}}{3 (\gamma+\tilde{\gamma})},\\
	\frac{1}{b}-\frac{\gamma}{3} \leq \frac{1}{\tilde{b}}-\frac{\tilde{\gamma}}{3}
	<\frac{1}{b}-\frac{\gamma}{3}+\frac{1}{3}, \qquad
	\frac{3}{2}\lt( \frac{1}{\tilde{b}}-\frac{\tilde{\gamma}}{3}-\frac{1}{b}+\frac{\gamma}{3}\rt)<\frac{1}{a} \leq \frac{1}{\tilde{a}'} \leq 1
\end{gather*}
for $0<\gamma, \tilde{\gamma}<1$.

When we interpolate this with case $(c)$, we will impose the restriction of $\tilde{\gamma}<1/2$ for the sake of simplicity. If one uses the interpolation argument with different weights, then it can be improved a little. But it is not affected to obtain the well-posedness result as a note.
Now we use the simple interpolation argument and so \eqref{eq2} holds for
\begin{gather}\label{i2}
\frac{1}{b}=\frac{(1-\theta)(\gamma+\varepsilon)}{3}+\frac{\theta}{b_2}, \quad \frac{1}{\tilde{b}}=\frac{(1-\theta)(1+\tilde{\gamma})}{3}+\frac{\theta}{\tilde{b}_2},\\
\label{ii2}
 \quad \frac{1}{a}=\frac{1-\theta}{2}+\frac{\theta}{a_2}, \quad \frac{1}{\tilde{a}}=\frac{1-\theta}{2}+\frac{\theta}{\tilde{a}_2}
\end{gather}
if $0\leq \theta\leq 1$ and it satisfies
\[
\begin{aligned}
\frac{1}{3} \lt( \frac{1}{\tilde{b}_2}-\frac{\tilde{\gamma}}{3}\rt) <& \frac{1}{b_2}-\frac{\gamma}{3} \leq -\frac{2-\gamma-\tilde{\gamma}}{ \gamma+\tilde{\gamma}}\lt( \frac{1}{\tilde{b}_2}-\frac{\tilde{\gamma}}{3}\rt)+\frac{3-\gamma-\tilde{\gamma}}{3 (\gamma+\tilde{\gamma})},\\
&	\frac{1}{b_2}-\frac{\gamma}{3} \leq \frac{1}{\tilde{b}_2}-\frac{\tilde{\gamma}}{3}
	<\frac{1}{b_2}-\frac{\gamma}{3}+\frac{1}{3},\\
	\frac{3}{2}&\lt( \frac{1}{\tilde{b}_2}-\frac{\tilde{\gamma}}{3}-\frac{1}{b_2}+\frac{\gamma}{3}\rt)<\frac{1}{a_2} \leq \frac{1}{\tilde{a}_2'} \leq 1 
\end{aligned}
\]
for $0<\tilde{\gamma}<1/2$ and $0<\gamma<1$.
Multiplying $\theta$ into these all, it follows that 
\begin{gather}
(\gamma+\tilde{\gamma})\lt( \frac{1}{b}-\frac{\gamma}{3}\rt)+(2-\gamma-\tilde{\gamma})\lt( \frac{1}{\tilde{b}}-\frac{\tilde{\gamma}}{3}-\frac{1}{3}\rt) < \frac{\theta}{3} <3\lt( \frac{1}{b}-\frac{\gamma}{3}\rt)+\frac{1}{3}-\frac{1}{\tilde{b}}+\frac{\tilde{\gamma}}{3} ,\nonumber \\
\frac{1}{b}-\frac{\gamma}{3} -\frac{1}{\tilde{b}}+\frac{\tilde{\gamma}}{3}+\frac{1}{3}< \frac{\theta}{3}, \qquad \frac{1}{\tilde{b}}-\frac{\tilde{\gamma}}{3} <\frac{1}{b}-\frac{\gamma}{3}+\frac{1}{3}, \nonumber\\
 \frac{3}{2} \lt(\frac{1}{\tilde{b}}-\frac{\tilde{\gamma}}{3}-\frac{1}{b}+\frac{\gamma}{3} \rt)<\frac{1}{a}\leq \frac{1}{\tilde{a}'}, \qquad \frac{1}{2}-\frac{1}{\tilde{a}}\leq \frac{\theta}{2}\nonumber
\end{gather}
by using \eqref{i2} and \eqref{ii2}.
Here, one sees that first condition appears as in \eqref{i03}.
Then all conditions of $\theta$ is summarized as follows;
\[
\begin{aligned}
	(\gamma+\tilde{\gamma})\lt( \frac{1}{b}-\frac{\gamma}{3}\rt)+(2-\gamma-\tilde{\gamma})\lt( \frac{1}{\tilde{b}}-\frac{\tilde{\gamma}}{3}-\frac{1}{3}\rt) &< \frac{\theta}{3} <3\lt( \frac{1}{b}-\frac{\gamma}{3}\rt)+\frac{1}{3}-\frac{1}{\tilde{b}}+\frac{\tilde{\gamma}}{3}\\
	\frac{1}{b}-\frac{\gamma}{3} -\frac{1}{\tilde{b}}+\frac{\tilde{\gamma}}{3}+\frac{1}{3}&< \frac{\theta}{3} \\
	\frac{2}{3} \lt(\frac{1}{2}-\frac{1}{\tilde{a}} \rt) &\leq \frac{\theta}{3}\\
		0&\leq \frac{\theta}{3} \leq \frac{1}{3}
\end{aligned}
\]
We compare each lower bound of $\theta/3$ to its upper bounds in turn. Starting from the first lower one, we have \eqref{i02} and $\frac{1}{\tilde{b}}-\frac{\tilde{\gamma}}{3}-\frac{1}{3}<\frac{1}{b}-\frac{\gamma}{3}$.
From the second lower one, we have
\[
\frac{\gamma}{3}<\frac{1}{b}, \qquad \frac{1}{b}-\frac{\gamma}{3}<\frac{1}{\tilde{b}}-\frac{\tilde{\gamma}}{3},
\]
and from the third lower one it follows that $0\leq 1/\tilde{a}$ (i.e., $1/\tilde{a}'\leq 1$) and the first condition of \eqref{i04}.
Finally, the last lower one implies
\[
\frac{1}{\tilde{b}}-\frac{\tilde{\gamma}}{3} <\frac{1}{3}+ 3\lt(\frac{1}{b}-\frac{\gamma}{3} \rt), 
\]
but it is automatically true.
Hence the proof is done.


\begin{thebibliography}{9}


\bibitem{AK} J. An, J. Kim, \textit{A note on the $H^s$-critical inhomogeneous nonlinear Schr\"odinger equation}, Z. Anal. Anwend. 42 (2023), 403–433.

\bibitem{AT} L. Aloui, S. Tayachi,
\textit{Local well-posedness for the inhomogeneous nonlinear Schr\"odinger equation}, Discrete Contin. Dyn. Syst. 41 (2021), 5409–5437.


\bibitem{B} L. Berg\'e, \textit{Soliton stability versus collapse}, Phys. Rev. E, 62 (2000), 3071-3074.

\bibitem{B2} J. Bourgain, \textit{Global wellposedness of defocusing critical nonlinear Schr\"odinger equation in the radial case}, J. Amer. Math. Soc. 12 (1999), no. 1, 145–171.


\bibitem{BH} J. J. Benedetto and H. P. Heinig, \textit{Weighted Fourier inequalities: new proofs and generalizations}, J. Fourier Anal. Appl., 9 (2003), 1-37.
		
\bibitem{BL} J. Bergh and J. L\"ofstr\"om, \textit{Interpolation Spaces, An Introduction}, Springer-Verlag, Berlin-New York, 1976.
	
\bibitem{BPVT} J. Belmonte-Beitia, V. M. P\'erez-Garcia, V. Vekslerchik, and P. J. Torres, \textit{Lie Symmetries and Solitons in Nonlinear Systems with Spatially Inhomogeneous Nonlinearities.} Phys. Rev. Lett. 98, 064102.

	

\bibitem{CCF} L. Campos, S. Correia, L. G. Farah \textit{Some well-posedness and ill-posedness results for the INLS equation}, preprint. 


\bibitem{CL} Y. Cho and K. Lee, \textit{On the focusing energy-critical inhomogeneous NLS: weighted space approach}, Nonlinear Anal. 205 (2021), Paper No. 112261, 21 pp.



\bibitem{CHL} Y. Cho, S. Hong and K. Lee, \textit{On the global well-posedness of focusing energy-critical inhomogeneous NLS}, J. Evol. Equ. 20 (2020), no. 4, 1349–1380.

\bibitem{CK} M. Christ and A. Kiselev, \textit{Maximal functions associated to filtrations}, J. Funct. Anal. 179 (2001), 409-425.


\bibitem{CKSTT} J. Colliander, M. Keel, G. Staffilani, H. Takaoka and T. Tao, \textit{Global well-posedness and scattering for the energy-critical nonlinear Schr\"odinger equation in $\mathbb{R}^3$},  Ann. of Math. 167  (2008), 767-865.


\bibitem{CW} T. Cazenave and F. B. Weissler, \textit{Some remarks on the nonlinear Schr\"odinger equation in the critical case}, Nonlinear Semigroups, Partial differential equations and attractors (Washington, DC, 1987), Lecture Notes in Math. 1394, Springer, Berlin, 1989, 18-29.


\bibitem{F} L. G. Farah, \textit{Global well-posedness and blow-up on the energy space for the inhomogeneous nonlinear Schr\"odinger equation}, J. Evol. Equ. 16 (2016), 193-208.

\bibitem{F2} D. Foschi, \textit{Inhomogeneous Strichartz estimates}, J. Hyperbolic Differ. Equ. 2 (2005), no. 1, 1-24.


\bibitem{GM} C. M. Guzman, J. Murphy, \textit{Scattering for the non-radial energy-critical inhomogeneous NLS}, J. Differential Equations 295 (2021), 187–210.

\bibitem{GX} C. M. Guzman, C. Xu, \textit{Dynamics of the non-radial energy-critical inhomogeneous NLS}, arXiv:2406.07535, preprint



\bibitem{H} L.H\"ormander, \textit{Estimates for translation invariant operators in $L^p$ spaces}, Acta Math. 104 (1960), 93–140.


\bibitem{KLS} J. Kim, Y. Lee and I. Seo, \textit{On well-posedness for the inhomogeneous nonlinear Schr\"{o}dinger equation in the critical case}, J. Differential Equations 280 (2021), 179–202.
		
\bibitem{KLS2} S. Kim, Y. Lee and I. Seo, \textit{Sharp weighted Strichartz estimates and critical inhomogeneous Hartree equations},
Nonlinear Anal.240 (2024), Paper No. 113463, 17 pp.
		
\bibitem{KM} G. Kenig and F. Merle, \textit{Global well-posedness, scattering and blow-up for the energy-critical, focusing, non-linear Schr\"odinger in the radial case}, Invent. Math. 166 (2006), 645-675.
		
\bibitem{KS} Y. Koh and I. Seo, \textit{Inhomogeneous Strichartz estimates for Schr\"odinger's equation}, J. Math. Anal. Appl. 442 (2016), no. 2, 715–725.
			
		
	
\bibitem{LS} Y. Lee and I. Seo, \textit{The Cauchy problem for the energy-critical inhomogeneous nonlinear Schr\"{o}dinger equation}, Arch. Math. (Basel) 117 (2021), no. 4, 441–453.

\bibitem{LYZ} X.Liu, K. Yang, T. Zhang, \textit{Dynamics of threshold solutions for the energy-critical inhomogeneous NLS}, arXiv:2409.00073, preprint.

\bibitem{LZ} X. Liu, T. Zhang, \textit{Global existence, scattering and blow-up for the focusing energy-critical inhomogeneous NLS}, preprint.

\bibitem{OR} T. Ozawa and K. M. Rogers, \textit{Sharp Morawetz estimates}, J. Anal. Math. 121 (2013), 163-175.

\bibitem{P} D. Park, \textit{Global well-posedness and scattering of the defocusing energy-critical inhomogeneous nonlinear Schr\"odinger equation with radial data}, J. Math. Anal. Appl. 536 (2024), Paper No. 128202, 28 pp.

\bibitem{Pi} H. R. Pitt, \textit{Theorems on Fourier series and power series}, Duke Math. J., 3 (1937), 747-755.

\bibitem{S} E. M Stein, \textit{Singular integrals and differentiability properties of functions}, Princeton Mathematical Series, No. 30 Princeton University Press, Princeton, N.J. 1970 xiv+290 pp.

\bibitem{S2} E. M Stein, \textit{Harmonic analysis: real-variable methods, orthogonality, and oscillatory integrals}, Princeton Mathematical Series, No. 43 Princeton University Press, Princeton, N.J. 1993 xiv+695 pp.

\bibitem{S3} E. M. Stein, \textit{Interpolation of linear operators}, Trans. Amer. Math. Soc., 83 (1956), 482-492.


\bibitem{SW} E. M. Stein and G. Weiss, \textit{Fractional integrals on $n$-dimensional Euclidean space}, J. Math. Mech. 7 (1958), 503-514.

\bibitem{T} T. Tao, \textit{Global well-posedness and scattering for the higher-dimensional energy-critical nonlinear Schr\"odinger equation for radial data}, New York J. Math. 11 (2005), 57–80.



\bibitem{TV} T. Tao, M. Visan,\textit{Stability of energy-critical nonlinear Schr\"odinger equations in high dimensions}, Electron. J. Differential Equations(2005), No. 118, 28 pp.


\bibitem{V2} M. C. Vilela, \textit{Inhomogeneous Strichartz estimates for the Schr\"odinger equation}, Trans. Amer. Math. Soc. 359 (2007), no. 5, 2123–2136.




\end{thebibliography}
\end{document}